\documentclass{amsart}
\usepackage{amsmath}
\usepackage{amssymb}
\usepackage{amsthm}
\usepackage{mathtools}
\usepackage[usenames, dvipsnames]{color}

\DeclareMathOperator{\Tr}{Tr}
\DeclareMathOperator{\tr}{tr}

\DeclareMathOperator{\Vol}{Vol}
\DeclareMathOperator{\dvol}{dvol}

\DeclareMathOperator{\Ric}{Ric}

\DeclareMathOperator{\Rm}{Rm}

\DeclareMathOperator{\End}{End}

\DeclareMathOperator{\Int}{Int}

\newcommand{\oJ}{\overline{J}}

\newcommand{\oP}{\overline{P}}

\newcommand{\od}{\overline{d}}
\newcommand{\of}{\overline{f}}
\newcommand{\og}{\overline{g}}

\newcommand{\ow}{\overline{w}}

\newcommand{\odelta}{\overline{\delta}}

\newcommand{\oDelta}{\overline{\Delta}}

\newcommand{\onabla}{\overline{\nabla}}

\newcommand{\ch}{\widetilde{h}}
\newcommand{\cf}{\widetilde{f}}

\newcommand{\cphi}{\widetilde{\phi}}
\newcommand{\cpsi}{\widetilde{\psi}}

\newcommand{\hf}{\widehat{f}}
\newcommand{\hg}{\widehat{g}}
\newcommand{\hh}{\widehat{h}}
\newcommand{\hr}{\widehat{r}}

\newcommand{\hB}{\widehat{B}}
\newcommand{\hC}{\widehat{C}}

\newcommand{\hH}{\widehat{H}}

\newcommand{\hP}{\widehat{P}}

\newcommand{\hT}{\widehat{T}}
\newcommand{\hmB}{\widehat{\mathcal{B}}}

\newcommand{\hnabla}{\widehat{\nabla}}

\newcommand{\heta}{\widehat{\eta}}

\newcommand{\hphi}{\widehat{\phi}}
\newcommand{\hpsi}{\widehat{\psi}}

\newcommand{\lp}{\langle}
\newcommand{\rp}{\rangle}
\newcommand{\lv}{\lvert}
\newcommand{\rv}{\rvert}
\newcommand{\lV}{\lVert}
\newcommand{\rV}{\rVert}

\newcommand{\lots}{\mathrm{l.o.t.}}

\newcommand{\coronal}{coronal}

\newcommand{\mB}{\mathcal{B}}
\newcommand{\mC}{\mathcal{C}}

\newcommand{\mE}{\mathcal{E}}
\newcommand{\mF}{\mathcal{F}}

\newcommand{\mP}{\mathcal{P}}
\newcommand{\mQ}{\mathcal{Q}}

\newcommand{\bN}{\mathbb{N}}

\newcommand{\bR}{\mathbb{R}}

\newcommand{\suchthat}{\mathrel{}\middle|\mathrel{}}

\def\sideremark#1{\ifvmode\leavevmode\fi\vadjust{\vbox to0pt{\vss
 \hbox to 0pt{\hskip\hsize\hskip1em
 \vbox{\hsize3cm\tiny\raggedright\pretolerance10000
 \noindent #1\hfill}\hss}\vbox to8pt{\vfil}\vss}}}

\newcommand{\comment}[1]{}

\newtheorem{thm}{Theorem}[section]
\newtheorem{prop}[thm]{Proposition}
\newtheorem{lem}[thm]{Lemma}
\newtheorem{cor}[thm]{Corollary}

\theoremstyle{definition}
\newtheorem{defn}[thm]{Definition}

\theoremstyle{remark}
\newtheorem{remark}[thm]{Remark}

\numberwithin{equation}{section}

\begin{document}

\title{Boundary operators associated to the sixth-order GJMS operator}
\author{Jeffrey S. Case}
\address{109 McAllister Building \\ Penn State University \\ University Park, PA 16801}
\email{jscase@psu.edu}
\author{Weiyu Luo}
\address{Department of Electrical Engineering and Computer Science\\ University of California, Irvine\\ Irvine, CA 92617}
\email{weiyul7@uci.edu}
\begin{abstract}
We describe a set of conformally covariant boundary operators associated to the sixth-order GJMS operator on a conformally invariant class of manifolds which includes compactifications of Poincar\'e--Einstein manifolds.  This yields a conformally covariant energy functional for the sixth-order GJMS operator on such manifolds.  Our boundary operators also provide a new realization of the fractional GJMS operators of order one, three, and five as generalized Dirichlet-to-Neumann operators.  This allows us to prove some sharp Sobolev trace inequalities involving the interior $W^{3,2}$-seminorm, including an analogue of the Lebedev--Milin inequality on six-dimensional manifolds.
\end{abstract}
\maketitle

\section{Introduction}
\label{sec:intro}

The GJMS operators~\cite{GJMS1992} are conformally covariant differential operators with leading-order term an integer power of the Laplacian.  These operators play a key role in many questions at the intersection of geometry and analysis.  As one example, the GJMS operator of order two --- more commonly known as the conformal Laplacian --- controls the behavior of the scalar curvature within a conformal class and as such plays an important role in the resolution of the Yamabe Problem (see~\cite{LeeParker1987} and references therein).  As another example, the Sobolev embedding $W^{k,2}(\bR^n)\hookrightarrow L^{\frac{2n}{n-2k}}(\bR^n)$ can be seen as a consequence of the sharp Sobolev inequality
\begin{equation}
 \label{eqn:classical_sobolev_inequality}
 \int_{\bR^n} w\,(-\Delta)^kw \geq C_{n,k}\left( \int_{\bR^n} \lv w\rv^{\frac{2n}{n-2k}}\right)^{\frac{n-2k}{n}}
\end{equation}
for all $w\in W^{k,2}(\bR^n)$, where $C_{n,k}$ is an explicit constant and $(-\Delta)^k$ is the GJMS operator of order $2k$ on flat Euclidean space.   By conformal covariance, one can use stereographic projection to write~\eqref{eqn:classical_sobolev_inequality} as an equivalent inequality on the round $n$-sphere (cf.\ \cite{Beckner1993}).

In order to study a GJMS operator and its related scalar invariants on a manifold with boundary, one should first find conformally covariant boundary operators which are suitably adapted to the GJMS operator in question.  For the case of the conformal Laplacian, Cherrier~\cite{Cherrier1984} and Escobar~\cite{Escobar1988} showed that the trace (or restriction) operator and the conformal Robin operator serve as the appropriate boundary operators.  In particular, Escobar proved the sharp Sobolev trace inequality	
\begin{equation}
 \label{eqn:classical_sobolev_trace_inequality}
 \int_{\bR_+^{n+1}} w\,(-\Delta)w + \oint_{\bR^n} w\,(-\partial_y w) \geq D_{n,k}\left( \oint_{\bR^n} \lv w\rv^{\frac{2n}{n-1}} \right)^{\frac{2n}{n-1}} 
\end{equation}
for all $w\in W^{1,2}(\bR_+^{n+1})$, where $D_{n,k}$ is an explicit constant and $-\partial_y$ is the conformal Robin operator on flat Euclidean upper half space $\bR_+^{n+1}$.  The zeroth-order term of the conformal Robin operator is the mean curvature of the boundary, leading these operators to play an important role in the resolution of the boundary Yamabe problem (see, for example, \cite{Escobar1992a,Escobar1992}).

The conformal Laplace operator $L_2$ and conformal Robin operator $B_1^1$ naturally give rise to a conformally covariant Dirichlet-to-Neumann operator $\mB_1^1$ on the boundary of a Riemannian manifold for which $\ker L_2\cap\ker B_1^1=\{0\}$.  The operator $\mB_1^1$ recovers $(-\Delta)^{1/2}$ in Euclidean space.  On boundaries of asymptotically hyperbolic manifolds, there is another formally self-adjoint, conformally covariant pseudodifferential operator with leading term $(-\Delta)^{1/2}$, namely the fractional GJMS operator $P_1$ of order $1$; see~\cite{GrahamZworski2003}.  It turns out that $\mB_1^1=P_1$, provided the latter is defined in terms of the Loewner--Nirenberg metric~\cite{GuillarmouGuillope2007}.

The purpose of this article is to identify the boundary operators associated to the sixth-order GJMS operator and use them both to prove sharp Sobolev trace inequalites involving the $W^{3,2}$-seminorm and give a new realization of the fractional GJMS operators of order $1$, $3$, and $5$.  This work is motivated by recent developments in three directions.  First, boundary operators for the fourth-order GJMS operator --- more commonly known as the Paneitz operator --- and their relations to sharp Sobolev trace inequalities and fractional GJMS operators are now well understood~\cite{AcheChang2015,Case2015b,ChangQing1997a,GoverPeterson2018,Grant2003,Stafford2006}.  Second, this understanding of the Paneitz operator and its corresponding boundary operators is yielding new insights into the Yamabe-type problem for the fractional third-order $Q$-curvature~\cite{CaseChang2013}.  Third, an algorithmic approach to constructing boundary operators for the higher-order GJMS operators via tractor calculus has recently been established~\cite{BransonGover2001,GoverPeterson2018}.  Relative to this latter work, the benefits of our approach are (i) that it directly yields local formulas for the boundary operators which are valid in all settings where the sixth-order GJMS operator is defined and (ii) that the generalized Dirichlet-to-Neumann operators constructed by our method are automatically formally self-adjoint.  Unfortunately, unlike the tractor approach~\cite{BransonGover2001,GoverPeterson2018}, it does not seem practical to extend our method to GJMS operators of arbitrarily high order.

To describe our results, recall that the sixth-order GJMS operator of $(X^{n+1},g)$, $n\geq5$, is given by
\begin{multline}
 \label{eqn:L6}
 L_6 := -\Delta^3 + \Delta\delta\left((n-1)Jg-8P\right)d + \delta\left((n-1)Jg-8P\right)d\Delta \\ - \frac{n-1}{2}\Delta\left(J\Delta\right) - \delta T_4 d + \frac{n-5}{2}Q_6, 
\end{multline}
where $P$ is the Schouten tensor, $J$ is its trace, $B$ is the Bach tensor,
\begin{multline}
 \label{eqn:L6-T4}
 T_4 := \left(-(n-5)\Delta J + \frac{3n^2-6n-13}{4}J^2 - 4(n-3)\lv P\rv^2\right)g \\ - 8(n-1)JP + 48\left(P^2 + \frac{1}{3(n-3)}B\right)
\end{multline}
for $P^2$ the square of $P$ as an endomorphism, and $Q_6$ is the sixth-order $Q$-curvature
\begin{multline*}
 Q_6 := -\Delta^2 J - \frac{n-5}{2}J\Delta J - \frac{n+3}{2}\Delta J^2 + 4\Delta\lv P\rv^2 + 8\delta\left(P(\nabla J)\right) \\ + \frac{(n-1)(n+3)}{4}J^3 - 4(n+1)J\lv P\rv^2 + 16\left(\tr P^3 + \frac{1}{n-3}\lp B,P\rp\right) . 
\end{multline*}
We emphasize that~\eqref{eqn:L6} defines $L_6$ as an operator, so that the right-hand side consists of sums of compositions of operators.  See Section~\ref{sec:bg} for a more detailed explanation of our notation.  The conformal covariance of $L_6$ was proven independently by Branson~\cite{Branson1985} and W\"unsch~\cite{Wunsch1986}, though this can also be deduced via Juhl's recursive formula~\cite{Juhl2013} for the GJMS operators.  In fact, $L_6$ is well-defined so long as $\dim X$ is odd or $g$ is either locally conformally flat or Einstein.  However, not every four-dimensional manifold admits a sixth-order conformally covariant operator with leading-order term $(-\Delta)^3$; see~\cite{Graham1992}.

It is clear from~\eqref{eqn:L6} that $L_6$ is a sixth-order operator which is formally self-adjoint in the interior of $X$.  We thus expect there to be a set of six operators $B_j^5$, $0\leq j\leq 5$, of total and normal order $j$ which give rise to formally self-adjoint boundary problems for $L_6$.  The following theorem in fact gives a stronger statement about the boundary operators.

\begin{thm}
 \label{thm:boundary_operators}
 Let $(X^{n+1},g)$, $n\geq5$, be a compactification of a Poincar\'e--Einstein manifold $(X_0^{n+1},M^n,g_+)$.  There exist explicit operators $B_j^5\colon C^\infty(X)\to C^\infty(M)$,
 \begin{align*}
  B_0^5(u) & = u\rv_M, \\
  B_1^5(u) & = \eta u + \lots, \\
  B_2^5(u) & = \Delta u - \frac{4}{3}\oDelta u\rv_M + \lots, \\
  B_3^5(u) & = \eta\Delta u - 4\oDelta\eta u + \lots, \\
  B_4^5(u) & = -\oDelta^2 - 4\oDelta(\Delta u)\rv_M + 8\oDelta^2u\rv_M + \lots, \\
  B_5^5(u) & = \eta\Delta^2u + \frac{4}{3}\oDelta\eta\Delta u + \frac{8}{3}\oDelta^2\eta u + \lots,
 \end{align*}
 where $\eta$ denotes the outward-pointing normal vector field along the boundary, $\oDelta$ denotes the Laplacian defined in terms of the induced metric $\og:=g\rv_{TM}$ on the boundary, and ``$\,\lots$'' in $B_j^5$ denotes terms of order at most $j-1$ in $u$, such that
 \begin{enumerate}
  \item the operator $B_j^5$ is conformally covariant of bidegree $\bigl(-\frac{n-5}{2},-\frac{n+2j-5}{2}\bigr)$; i.e.
  \[ \hB_j^5(u) = e^{-\frac{n+2j-5}{2}\sigma}B_j^5\left(e^{\frac{n-5}{2}\sigma}u\right) \]
  for all $u,\sigma\in C^\infty(X)$, where $\hB_j^5$ is defined with respect to $\hg:=e^{2\sigma}g$; and
  \item the bilinear form $\mQ_6\colon C^\infty(X)\times C^\infty(X)\to\bR$,
  \[ \mQ_6(u,v) := \int_X u\,L_6v\,\dvol_g + \sum_{j=0}^2\oint_M B_j^5(u)\,B_{5-j}^5(v)\,\dvol_{\og} , \]
  is symmetric.
 \end{enumerate}
\end{thm}

In fact, we prove a stronger version of Theorem~\ref{thm:boundary_operators} which requires only that $L_6$ is defined and that the boundary satisfy certain conformally invariant assumptions involving only the extrinsic geometry of the boundary $\partial X$; see Section~\ref{sec:invariant} for this version and explicit formulae for the operators $B_j^5$.

By definition, $(L_6;B)$ is formally self-adjoint if $\int uL_6v=\int vL_6u$ for all $u,v\in\ker B$, where $B$ is a $3$-tuple of boundary operators.  It follows from Theorem~\ref{thm:boundary_operators} that each of the eight possible $3$-tuples $B$ formed by choosing an operator from each of $\{B_0^5,B_5^5\}$, $\{B_1^5,B_4^5\}$, and $\{B_2^5,B_3^5\}$ is such that $(L_6;B)$ is formally self-adjoint.  The operators constructed by Gover and Peterson~\cite{GoverPeterson2018}, which are defined whenever $L_6$ is defined and under no assumptions on the geometry of the boundary, also have the property that such triples $(L_6;B)$ are formally self-adjoint, though it is not yet known whether the corresponding bilinear form $\mQ_6$ is symmetric.  Similar operators constructed earlier by Branson and Gover~\cite{BransonGover2001} are such that the corresponding bilinear form $\mQ_6$ is symmetric, but their construction does not work in the critical dimension $n=5$.

One reason to desire the symmetry of $\mQ_6$, rather than just the formal self-adjointness of $(L_6;B)$, is that it implies the formal self-adjointness of the generalized Dirichlet-to-Neumann operators associated to  $L_6$ and its boundary operators from Theorem~\ref{thm:boundary_operators}.  As we show in Proposition~\ref{prop:dirichlet-to-neumann-operators} below, under the (conformally invariant) assumption that
\[ \ker L_6 \cap \ker B_0^5 \cap \ker B_1^5 \cap \ker B_2^5 = \{0\}, \]
for any triple $(f,\phi,\psi)\in\bigl(C^\infty(M)\bigr)^3$, there is a unique solution $u_{f,\phi,\psi}\in C^\infty(X)$ of
\begin{equation}
 \label{eqn:L6_extension}
 \begin{cases}
  L_6(u) = 0, & \text{in $X$}, \\
  B_0^5(u) = f, & \text{on $M$}, \\
  B_1^5(u) = \phi, & \text{on $M$}, \\
  B_2^5(u) = \psi, & \text{on $M$}.
 \end{cases}
\end{equation}
In particular, the generalized Dirichlet-to-Neumann operators $\mB_5^5(f):=B_5^5(u_{f,0,0})$, $\mB_3^5(\phi):=B_4^5(u_{0,\phi,0})$, and $\mB_1^5(\psi):=B_3^5(u_{0,0,\psi})$ are well-defined, and the symmetry of $\mQ_6$ implies that these operators are formally self-adjoint.  Indeed, $\mB_j^5$, $j\in\{1,3,5\}$, is also conformally covariant with leading order term a multiple of $(-\Delta)^{j/2}$; see Proposition~\ref{prop:dirichlet-to-neumann-operators}.

The operators $\mB_j^5$ constructed above have the same properties as the fractional GJMS operators $P_j$ constructed by Graham and Zworski~\cite{GrahamZworski2003}, leading one to wonder how these operators are related.  When $(X^{n+1},g)$ is a compactification of a Poincar\'e--Einstein manifold, it turns out that $\mB_j^5$ and $P_j$ are proportional.  Indeed, even more is true:

\begin{thm}
 \label{thm:dirichlet-to-neumann}
 Let $(X^{n+1},g)$ be a compactification of a Poincar\'e--Einstein manifold $(X_0^{n+1},M^n,g_+)$ such that $\frac{n^2}{4}-\gamma^2\not\in\sigma_{pp}(-\Delta_{g_+})$ for $\gamma\in\{1/2,3/2,5/2\}$.  Suppose additionally that $u\in C^\infty(X)$ is such that $L_6u=0$.  Then
 \begin{align*}
  B_5^5(u) & = \frac{8}{3}P_5\left(B_0^5(u)\right), \\
  B_4^5(u) & = 8P_3\left(B_1^5(u)\right), \\
  B_3^5(u) & = 3P_1\left(B_2^5(u)\right),
 \end{align*}
 where $B_j^5$, $0\leq j\leq 5$, are the boundary operators of Theorem~\ref{thm:boundary_operators} and $P_{2\gamma}$ are the fractional GJMS operators of order $2\gamma$.
\end{thm}

In other words, for compactifications of Poincar\'e--Einstein manifolds, it holds that $B_5^5(u_{f,0,0})=B_5^5(u_{f,\phi,\psi})$ for all $f,\phi,\psi\in C^\infty(M)$.  The proof of Theorem~\ref{thm:dirichlet-to-neumann} uses heavily the fact that the GJMS operators factor at Einstein metrics~\cite{FeffermanGraham2012,Gover2006q}.  We do not know if the Poincar\'e--Einstein assumption can be relaxed.  Note also that requiring $u\in\ker B_1^5\cap\ker B_2^5$ in Theorem~\ref{thm:dirichlet-to-neumann} yields a curved analogue of the higher-order Caffarelli--Silvestre-type extension theorem of R.\ Yang~\cite{ChangYang2017}.

Another reason to desire the symmetry of $\mQ_6$ in Theorem~\ref{thm:boundary_operators} is that it gives rise to variational characterizations of solutions of $L_6u=v$ with various boundary conditions.  For example, a function $u\in C^\infty(X)$ is a solution of~\eqref{eqn:L6_extension} if and only if it is a critical point of the functional
\[ u \mapsto \mE_6(u) := \mQ_6(u,u) \]
when constrained to the set
\begin{equation}
 \label{eqn:mCfphipsi}
 \mC_{f,\phi,\psi} := \left\{ u\in C^\infty(M) \suchthat B_0^5(u)=f, B_1^5(u)=\phi, B_2^5(u)=\psi \right\} .
\end{equation}
Under an additional spectral assumption on the Laplacian of the Poincar\'e--Einstein metric $g_+$, one can in fact minimize the functional $\mE_6$ in $\mC_{f,\phi,\psi}$.

\begin{thm}
 \label{thm:L2-trace}
 Let $(X^{n+1},g)$ be a compactification of a Poincar\'e--Einstein manifold $(X_0^{n+1},M^n,g_+)$ such that $\lambda_1(-\Delta_{g_+}) > \frac{n^2-1}{4}$.  Given any $f,\phi,\psi\in C^\infty(M)$, it holds that
 \begin{equation}
  \label{eqn:L2-trace}
  \mE_6(u) \geq \oint_M \left( \frac{8}{3}f\,P_5f + 8\phi\,P_3\phi + 3\psi\,P_1\psi \right) \,\dvol_{g\rv_{TM}} 
 \end{equation}
 for all $u\in\mC_{f,\phi,\psi}$.  Moreover, equality holds in~\eqref{eqn:L2-trace} if and only if $u$ is the unique solution of~\eqref{eqn:L6_extension}.
\end{thm}

Note that the spectral assumption of Thoerem~\ref{thm:L2-trace} holds automatically when the conformal boundary has nonnegative Yamabe constant~\cite{Lee1995}.

In Section~\ref{sec:traces}, we prove a more general version of Theorem~\ref{thm:L2-trace} which only requires conformally invariant assumptions on the spectrum of $L_6$ and the extrinsic geometry of the boundary $\partial X$.

Theorem~\ref{thm:L2-trace} gives a sharp norm inequality for the well-known embedding
\begin{equation}
 \label{eqn:trace_embedding}
 \Tr \colon W^{3,2}(X) \hookrightarrow W^{5/2,2}(\partial X) \oplus W^{3/2,2}(\partial X) \oplus W^{1/2,2}(\partial X) 
\end{equation}
as well as an explicit right inverse.  By combining~\eqref{eqn:trace_embedding} with the Sobolev embedding $W^{k,2}(M^n)\hookrightarrow L^{\frac{2n}{n-2k}}(M^n)$, $n>2k$, one obtains the embedding
\begin{equation}
 \label{eqn:Lp-trace-embedding}
 W^{3,2}(X^{n+1}) \hookrightarrow L^{\frac{2n}{n-5}}(\partial X) \oplus L^{\frac{2n}{n-3}}(\partial X) \oplus L^{\frac{2n}{n-1}}(\partial X) .
\end{equation}
One can deduce a \emph{sharp} norm inequality for the embedding~\eqref{eqn:Lp-trace-embedding} from Theorem~\ref{thm:L2-trace} and a sharp norm inequality for the embedding $W^{k,2}(M^n)\hookrightarrow L^{\frac{2n}{n-2k}}(M^n)$.  Three particular cases of interest are upper half space, a closed Euclidean ball, and a round hemisphere.

\begin{cor}
 \label{cor:Lp-trace-upper}
 Let $\bR_+^{n+1}$ denote the (closed) upper half space
 \[ \bR_+^{n+1} = \left\{ (x,y) \in \bR^n \times [0,\infty) \right\} \]
 equipped with the Euclidean metric.  For all $u\in C^\infty(\bR_+^{n+1}) \cap W^{3,2}(\bR_+^{n+1})$, it holds that
 \begin{multline*}
  \frac{8}{3}C_{n,5/2}\lV f\rV_{\frac{2n}{n-5}}^2 + 8C_{n,3/2}\lV\phi\rV_{\frac{2n}{n-3}}^2 + 3C_{n,1/2}\lV\psi\rV_{\frac{2n}{n-1}}^2 \\ \leq \int_{\bR_+^{n+1}} \lv\nabla\Delta u\rv^2 + \oint_{\partial\bR_+^{n+1}} \left\{ 8\lp\onabla\psi,\onabla\phi\rp + \frac{16}{3}(\oDelta\phi)(\oDelta f) \right\} ,
 \end{multline*}
 where the $L^p$-norms on the left-hand side are taken with respect to the Lebesgue measure on $\bR^n=\partial\bR_+^{n+1}$,
 \begin{align*}
  f(x) & = u(x,0), \\
  \phi(x) & = -\frac{\partial u}{\partial y}(x,0), \\
  \psi(x) & = \frac{\partial^2u}{\partial y^2}(x,0) - \frac{1}{3}\oDelta u(x,0)
 \end{align*}
 for all $x\in\bR^n$, and
 \begin{equation}
  \label{eqn:sobolev_constants}
  C_{n,\gamma} = \frac{\Gamma\bigl(\frac{n+2\gamma}{2}\bigr)}{\Gamma\bigl(\frac{n-2\gamma}{2}\bigr)}\Vol(S^n)^{\frac{2\gamma}{n}} .
 \end{equation}
 Moreover, equality holds if and only if $\Delta^3u=0$ and there are points $x_1,x_2,x_3\in\bR^n$, constants $a_1,a_2,a_3\in\bR$, and positive constants $\varepsilon_1,\varepsilon_2,\varepsilon_3\in\bR$ such that
 \begin{equation}
  \label{eqn:Lp-trace-upper-rigidity} 
  \begin{split}
   f(x) & = a_1\left(\varepsilon_1 + \lv x-x_1\rv^2\right)^{-\frac{n-5}{2}} , \\
   \phi(x) & = a_2\left(\varepsilon_2 + \lv x-x_2\rv^2\right)^{-\frac{n-3}{2}} , \\
   \psi(x) & = a_3\left(\varepsilon_3 + \lv x-x_3\rv^2\right)^{-\frac{n-1}{2}}  
  \end{split}
 \end{equation}
 for all $x\in\bR^n$.
\end{cor}

\begin{cor}
 \label{cor:Lp-trace-disk}
 Let $B^{n+1}\subset\bR^{n+1}$ denote the closed unit ball
 \[ B^{n+1} = \left\{ x\in\bR^{n+1} \suchthat \lv x\rv^2\leq 1 \right\} \]
 equipped with the Euclidean metric.  For all $u\in C^\infty(B^{n+1})$, it holds that
 \begin{align*}
  \MoveEqLeft[2] \frac{8}{3}C_{n,5/2}\lV f\rV_{\frac{2n}{n-5}}^2 + 8C_{n,3/2}\lV\phi\rV_{\frac{2n}{n-3}}^2 + 3C_{n,1/2}\lV\psi\rV_{\frac{2n}{n-1}}^2 \\
   & \leq \int_{B^{n+1}} \lv\nabla\Delta u\rv^2 + \oint_{\partial B^{n+1}} \biggl\{ \frac{n-9}{2}\psi^2 + 8\lp\onabla\psi,\onabla\phi\rp + 2(n^2-9)\psi\phi \\
   & \quad - \frac{4(n-3)}{3}\lp\onabla\psi,\onabla f\rp - \frac{(n-3)(n-5)(n+3)}{3}f\psi + 8(n-3)\phi^2 \\
   & \quad + \frac{16}{3}(\oDelta\phi)(\oDelta f) + \frac{8(n^2-4n-3)}{3}\lp\onabla\phi,\onabla f\rp \\
   & \quad + \frac{(n-5)(n-3)^2(n+3)}{3}\phi f + \frac{8(n+3)}{9}(\oDelta f)^2 \\
   & \quad + \frac{4(n^3+n^2-21n-9)}{9}\lv\onabla f\rv^2 + \frac{(n-5)(n-3)(n+3)(n^2+4n-9)}{18}f^2 \biggr\} ,
 \end{align*}
 where
 \begin{align*}
  f & = u\rv_{\partial B^{n+1}}, \\
  \phi & = \left.\frac{\partial u}{\partial r}\right|_{r=1} + \frac{n-5}{2}f, \\
  \psi & = \left.\frac{\partial^2u}{\partial r^2}\right|_{r=1} + (n-4)\phi - \frac{1}{3}\oDelta f - \frac{(n-5)(n-6)}{6}f
 \end{align*}
 for $r$ the distance to $0\in B^{n+1}$.  Moreover, equality holds if and only if $\Delta^3u=0$ and there are constants $a_1,a_2,a_3\in\bR$ and points $x_1,x_2,x_3\in\Int(B^{n+1})$ such that
 \begin{equation}
  \label{eqn:Lp-trace-disk-rigidity}
  \begin{split}
   f(x) & = a_1\left(1+x\cdot x_1\right)^{-\frac{n-5}{2}}, \\
   \phi(x) & = a_2\left(1 + x\cdot x_2\right)^{-\frac{n-3}{2}}, \\
   \psi(x) & = a_3\left(1 + x\cdot x_3\right)^{-\frac{n-1}{2}}
  \end{split}
 \end{equation}
 for all $x\in\partial B^{n+1}$.
\end{cor}

\begin{cor}
 \label{cor:Lp-trace-hemisphere}
 Let $S_+^{n+1}$ denote the closed upper hemisphere
 \[ S_+^{n+1} = \left\{ x=(x_0,\dotsc,x_{n+1}) \in \bR^{n+2} \suchthat \lv x\rv=1, x_{n+1}\geq0 \right\} \]
 equipped with the round metric induced by the Euclidean metric on $\bR^{n+2}$.  For all $u\in C^\infty(S_+^{n+1})$, it holds that
 \begin{align*}
  \MoveEqLeft[1] \frac{8}{3}C_{n,5/2}\lV f\rV_{\frac{2n}{n-5}}^2 + 8C_{n,3/2}\lV\phi\rV_{\frac{2n}{n-3}}^2 + 3C_{n,1/2}\lV\psi\rV_{\frac{2n}{n-1}}^2 \\
  & \leq \int_{S_+^{n+1}} \biggl\{ \lv\nabla\Delta u\rv^2 + \frac{3n^2-35}{4}(\Delta u)^2 + \frac{3n^4-70n^2+259}{16}\lv\nabla u\rv^2 + \frac{\Gamma\bigl(\frac{n+7}{2}\bigr)}{\Gamma\bigl(\frac{n-5}{2}\bigr)}u^2 \biggr\} \\
   & \quad + \oint_{\partial S_+^{n+1}} \biggl\{ 8\lp\onabla\psi,\onabla\phi\rp + \frac{3n^2-8n+13}{2}\psi\phi + \frac{16}{3}(\oDelta f)(\oDelta\phi) \\
   & \qquad\qquad + \frac{2(5n^2-8n-37)}{3}\lp\onabla f,\onabla\phi\rp + \frac{(n-3)(n-5)(3n^2+4n-11)}{12}f\phi \biggr\} ,
 \end{align*}
 where
 \begin{align*}
  f & = u\rv_{\partial S_+^{n+1}}, \\
  \phi & = \eta u, \\
  \psi & = \Delta u - \frac{4}{3}\oDelta f + \frac{(n-3)(n-5)}{12}f
 \end{align*}
 for $\eta=-\partial_{n+1}$ the outward-pointing unit normal along $\partial S_+^{n+1}$.  Moreover, equality holds if and only if
 \[ \left(-\Delta + \frac{(n+1)(n-1)}{4}\right)\left(-\Delta + \frac{(n+3)(n-3)}{4}\right)\left(-\Delta + \frac{(n+5)(n-5)}{4}\right)u = 0 \]
 and there are constants $a_1,a_2,a_3\in\bR$ and points
 \[ \xi_1,\xi_2,\xi_3 \in \left\{ x=(x_0,\dotsc,x_{n+1})\in\bR^{n+2} \suchthat \lv x\rv<1, x_{n+1}=0 \right\} \]
 such that
 \begin{equation}
  \label{eqn:Lp-trace-hemisphere-rigidity}
  \begin{split}
   f(x) & = a_1\left(1+x\cdot \xi_1\right)^{-\frac{n-5}{2}}, \\
   \phi(x) & = a_2\left(1 + x\cdot \xi_2\right)^{-\frac{n-3}{2}}, \\
   \psi(x) & = a_3\left(1 + x\cdot \xi_3\right)^{-\frac{n-1}{2}}
  \end{split}
 \end{equation}
 for all $x\in\partial S_{+}^{n+1}$.
\end{cor}

Corollary~\ref{cor:Lp-trace-upper}, Corollary~\ref{cor:Lp-trace-disk} and Corollary~\ref{cor:Lp-trace-hemisphere} are equivalent by stereographic projection, though it is useful to have them written out explicitly in all cases.  Since sharp Sobolev (trace) inequalities are useful when studying (boundary) Yamabe problems (e.g.\ \cite{Aubin1976,Escobar1992a,GonzalezQing2010}), we expect these corollaries to find applications in studies of the higher-order fractional Yamabe problem (cf.\ \cite{GonzalezQing2010}).

One replacement of the embedding $W^{k,2}(M^n)\hookrightarrow L^{\frac{2n}{n-2k}}(M^n)$ in the critical case $n=2k$ is the Orlicz embedding $W^{k,2}(M^{2k})\hookrightarrow e^L(M^{2k})$.  There is a sharp Onofri-type inequality~\cite{Beckner1993} which establishes this embedding.  The critical cases of Corollary~\ref{cor:Lp-trace-upper}, Corollary~\ref{cor:Lp-trace-disk} and Corollary~\ref{cor:Lp-trace-hemisphere} are as follows:

\begin{cor}
 \label{cor:Lp-trace-critical-upper}
 Let $(\bR_+^6,dx^2+dy^2)$ denote the closed upper half space.  For all $u\in C^\infty(\bR_+^6)\cap W^{3,2}(\bR_+^6)$, it holds that
 \begin{multline*}
  3C_{5,1/2}\lV\psi\rV_{5/2}^2 + 8C_{5,3/2}\lV\phi\rV_5^2 + \frac{128}{5}\Vol(S^5)\ln\oint_{\partial\bR_+^6} e^{5(f-\bar f)}d\mu \\ \leq \int_{\bR_+^6}\lv\nabla\Delta u\rv^2 + \oint_{\partial\bR_+^6} \left\{ 8\lp\onabla\psi,\onabla\phi\rp + \frac{16}{3}(\oDelta\phi)(\oDelta f) \right\} ,
 \end{multline*}
 where $d\mu=\frac{1}{\Vol(S^5)}\bigl(\frac{1+\lv x\rv^2}{2}\bigr)^{-5}\dvol_{dx^2}$,
 \begin{align*}
  f(x) & = u(x,0), \\
  \phi(x) & = -\frac{\partial u}{\partial y}(x,0), \\
  \psi(x) & = \frac{\partial^2u}{\partial y^2}(x,0) - \frac{1}{3}\oDelta u(x,0)
 \end{align*}
 for all $x\in\bR^5=\partial\bR_+^6$, the $L^p$-norms on the left-hand side are taken with respect to the Riemannian volume element of $(\bR^5,dx^2)$, and $\of:=(\oint f\,d\mu)/(\oint d\mu)$ is the average of $f$ with respect to $d\mu$.  Moreover, equality holds if and only if $\Delta^3u=0$ and there are points $x_1,x_2,x_3\in\bR^5$, constants $a_1,a_2,a_3\in\bR$, and positive constants $\varepsilon_1,\varepsilon_2,\varepsilon_3\in\bR$ such that
 \begin{align*}
  f(x) & = a_1 - \ln\left(\varepsilon_1 +\lv x-x_1\rv^2\right) + \ln\left(1+\lv x\rv^2\right) , \\
  \phi(x) & = a_2\left(\varepsilon_2 + \lv x-x_2\rv^2\right)^{-1} , \\
  \psi(x) & = a_3\left(\varepsilon_3 + \lv x-x_3\rv^2\right)^{-2}
 \end{align*}
 for all $x\in\bR^5$.
\end{cor}

\begin{cor}
 \label{cor:Lp-trace-critical-disk}
 Let $(B^6,dx^2)$ denote the closed Euclidean unit ball.  For all $u\in C^\infty(B^6)$, it holds that
 \begin{align*}
  \MoveEqLeft[2] 3C_{5,1/2}\lV\psi\rV_{5/2}^2 + 8C_{5,3/2}\lV\phi\rV_5^2 + \frac{128}{5}\Vol(S^5)\ln\oint_{\partial B^6} e^{5(f-\of)}\,d\mu \\
   & \leq \int_{B^6} \lv\nabla\Delta u\rv^2 + \oint_{\partial B^6} \biggl\{ -2\psi^2 + 8\lp\onabla\psi,\onabla\phi\rp + 32\psi\phi - \frac{8}{3}\lp\onabla\psi,\onabla f\rp  \\
   & \qquad\qquad + 16\phi^2 + \frac{16}{3}(\oDelta\phi)(\oDelta f) + \frac{16}{3}\lp\onabla\phi,\onabla f\rp + \frac{64}{9}(\oDelta f)^2 + 16\lv\onabla f\rv^2 \biggr\} ,
 \end{align*}
 where $d\mu:=\frac{1}{\Vol(S^5)}\dvol_{d\theta^2}$ is the probability measure on $\partial B^6$ induced by $dx^2$,
 \begin{align*}
  f & = u\rv_{\partial B^6}, \\
  \phi & = \left.\frac{\partial u}{\partial r}\right|_{r=1}, \\
  \psi & = \left.\frac{\partial^2u}{\partial r^2}\right|_{r=1} + \phi - \frac{1}{3}\oDelta f
 \end{align*}
 for $r$ the distance to $0\in B^6$, the $L^p$-norms on the right-hand side are taken with respect to the Riemannian volume element $\dvol_{d\theta^2}$ of $(S^5,d\theta^2)$, and $\of:=(\oint f)/(\oint 1)$ is the average of $f$ with respect to $\dvol_{d\theta^2}$.  Moreover, equality holds if and only if there are constants $a_1,a_2,a_3\in\bR$ and points $x_1,x_2,x_3\in\Int(B^6)$ such that
 \begin{align*}
  f(x) & = a_1 - \ln\left( 1 + x\cdot x_1\right) , \\
  \phi(x) & = a_2\left(1 + x\cdot x_2\right)^{-1}, \\
  \psi(x) & = a_3\left(1 + x\cdot x_3\right)^{-2}
 \end{align*}
 for all $x\in\partial B^6$.
\end{cor}

\begin{cor}
 \label{cor:Lp-trace-critical-hemisphere}
 Let $(S_+^6,d\theta^2)$ denote the closed hemisphere.  For all $u\in C^\infty(S_+^6)$, it holds that
 \begin{align*}
  \MoveEqLeft[2] 3C_{5,1/2}\lV\psi\rV_{5/2}^2 + 8C_{5,3/2}\lV\phi\rV_5^2 + \frac{128}{5}\Vol(S^5)\ln\oint_{\partial S_+^6} e^{5(f-\of)}\,d\mu \\
   & \leq \int_{S_+^6} \left\{ \lv\nabla\Delta u\rv^2 + 10(\Delta u)^2 + 24\lv\nabla u\rv^2\right\} \\
   & \quad + \oint_{\partial S_+^6} \left\{ 8\lp\onabla\psi,\onabla\phi\rp + 24\psi\phi + \frac{16}{3}(\oDelta\phi)(\oDelta f) + 32\lp\onabla f,\onabla\phi\rp \right\} ,
 \end{align*}
 where $d\mu:=\frac{1}{\Vol(S^5)}\dvol$ is the probability measure on $\partial S_+^6$ induced by $d\theta^2$,
 \begin{align*}
  f & = u\rv_{\partial S_+^6}, \\
  \phi & = \eta u, \\
  \psi & = \Delta u - \frac{4}{3}\oDelta f
 \end{align*}
 for $\eta = -\partial_6$ the outward-pointing normal along $\partial S_+^6$.  Moreover, equality holds if and only if
 \[ \left(-\Delta + 6\right)\left(-\Delta + 4\right)\left(-\Delta\right)u = 0 \]
 and there are constants $a_1,a_2,a_3\in\bR$ and points
 \[ \xi_1,\xi_2,\xi_3 \in \left\{ x=(x_0,\dotsc,x_{n+1})\in\bR^{n+2}\suchthat \lv x\rv < 1, x_{n+1} = 0 \right\} \]
 such that
 \begin{align*}
  f(x) & = a_1 - \ln\left( 1 + x\cdot \xi_1\right) , \\
  \phi(x) & = a_2\left(1 + x\cdot \xi_2\right)^{-1}, \\
  \psi(x) & = a_3\left(1 + x\cdot \xi_3\right)^{-2}
 \end{align*}
 for all $x\in\partial S_+^6$.
\end{cor}

These inequalities generalize the Lebedev--Milin inequality for closed surfaces with boundary~\cite{OsgoodPhillipsSarnak1988} and closed four-manifolds with boundary~\cite{AcheChang2015,Case2015b}.  For this reason, we expect Corollary~\ref{cor:Lp-trace-critical-upper}, Corollary~\ref{cor:Lp-trace-critical-disk} and Corollary~\ref{cor:Lp-trace-critical-hemisphere} to be useful when studying variational problems involving the functional determinant (cf.\ \cite{ChangQing1997b,OsgoodPhillipsSarnak1988}) and the fifth-order fractional $Q$-curvature (cf.\ \cite{AcheChang2015,Case2015b}) on six-manifolds with boundary.

This article is organized as follows:

In Section~\ref{sec:bg} we recall the definitions of Poincar\'e--Einstein manifolds and the fractional GJMS operators, pointing out in particular some useful conformally invariant properties of the boundaries of compactifications of Poincar\'e--Einstein manifolds.  We also recall Branson's method~\cite{Branson1995} for finding conformally covariant operators and give some computational lemmas which are useful in proving the conformal covariance of our boundary operators.

In Section~\ref{sec:invariant} we give the full formulas for our boundary operators and prove Theorem~\ref{thm:boundary_operators}.  We also discuss the pseudodifferential operators they determine on the boundary.

In Section~\ref{sec:poincare} we prove Theorem~\ref{thm:dirichlet-to-neumann}.

In Section~\ref{sec:traces} we study the Sobolev trace embeddings~\eqref{eqn:trace_embedding} and~\eqref{eqn:Lp-trace-embedding} in dimensions at least six.  This includes deriving explicit sharp norm inequalities in the Euclidean upper half space, the Euclidean ball, and the round hemisphere.

\subsection*{Acknowledgments}

JSC was supported by a grant from the Simons Foundation (Grant No.\ 524601).

\section{Background}
\label{sec:bg}

\subsection{Important tensors on Riemannian manifolds with boundary}

We begin by recalling some important tensors defined on a Riemannian manifold $(X^{n+1},g)$.  The \emph{Schouten tensor} of $g$ is
\[ P := \frac{1}{n-1}\left(\Ric - \frac{R}{2n}g\right), \]
where $\Ric$ is the Ricci tensor and $R:=\tr_g\Ric$ is the scalar curvature of $g$.  We denote $J:=\tr_gP$, so that $J=R/2n$ is a constant multiple of the scalar curvature.  The significance of the Schouten tensor comes from the decomposition
\[ \Rm = W + P \wedge g \]
of the Riemann curvature tensor into the totally trace-free \emph{Weyl tensor} $W$ and the Kulkarni--Nomizu product $P\wedge g$ of the Schouten tensor and the metric.  We sometimes use abstract index notation to represent tensors; for example,
\[ (P\wedge g)_{ijkl} := P_{ik}g_{jl} + P_{jl}g_{ik} - P_{il}g_{jk} - P_{jk}g_{il} . \]
We use the metric $g$ to raise and lower indices.  For example, $P_i^j:=g^{jk}P_{ik}$, which is regarded as a section of either $\End(TX)$ or $\End(T^\ast X)$, depending on context.  We denote by $P^2$ the composition of $P$ with itself; i.e.
\[ (P^2)_i^j := P_i^kP_k^j, \]
and similarly for other compositions.  The \emph{Cotton tensor} is
\[ C_{ijk} := \nabla_iP_{jk} - \nabla_jP_{ik} \]
and the \emph{Bach tensor} is
\[ B_{ij} := \nabla^kC_{kij} + W_{ikjl}P^{kl}, \]
where $\nabla$ denotes the Levi-Civita connection of $g$.  Recall that each of $W_{ijkl}$, $C_{ijk}$ and $B_{ij}$ are trace-free.  Finally, the \emph{divergence $\delta$} of a $(0,2)$-tensor field $T_{ij}$ is
\[ (\delta T)_j := \nabla^i T_{ij}, \]
with a similar definition for the divergence of tensor fields of different rank.

Now let $(M^n,\og):=(\partial X,g\rv_{TM})$ denote the boundary of $X$ with the metric induced by $g$.  Riemannian tensors denoted by bars are defined with respect to $\og$; e.g.\ $\oP$ denotes the Schouten tensor of the induced metric on the boundary.  We sometimes use $h$ to denote the induced metric $\og$, though in this case still use bars to denote Riemannian invariants associated to $h$.  We denote by $\eta$ the outward-pointing unit normal along $M$.  The \emph{second fundamental form} of $M$ is defined by
\[ A(X,Y) := g\left(\nabla_X\eta,Y\right) \]
for all sections $X,Y$ of $TM$, and the \emph{mean curvature} of $M$ is $H:=\tr_{\og}A$.  The \emph{trace-free part of the second fundamental form} is $A_0:=A-\frac{1}{n}H\og$.

The following lemma collects some useful and well-known relationships between the extrinsic and intrinsic geometry of the boundary of a Riemannian manifold; see~\cite{Case2015b} for proofs.

\begin{lem}
 \label{lem:gauss-codazzi}
 Let $(X^{n+1},g)$ be a Riemannian manifold with boundary.  Along $M$ it holds that
 \begin{align*}
  P\rv_{TM} & = \oP - \frac{1}{n}HA_0 - \frac{1}{2n^2}H^2\og + \mF, \\
  J & = \oJ + P(\eta,\eta) - \frac{1}{2n}H^2 + \frac{1}{2(n-1)}\lv A_0\rv^2, \\
  P(\eta,\cdot)\rv_{TM} & = -\frac{1}{n}\onabla H + \frac{1}{n-1}\odelta A_0,
 \end{align*}
 where
 \begin{equation}
  \label{eqn:fialkow}
  \mF := \frac{1}{n-2}\left(W(\eta,\cdot,\eta,\cdot) + A_0^2 - \frac{1}{2(n-1)}\lv A_0\rv^2\og\right) .
 \end{equation}
 Moreover, given any $u\in C^\infty(X)$, along $M$ it holds that
 \begin{align*}
  \Delta u & = \nabla^2u(\eta,\eta) + \oDelta u + H\eta u, \\
  \nabla^2u(\eta,\cdot)\rv_{TM} & = \onabla\eta u - A(\onabla u,\cdot) , \\
  \eta\Delta u & = \nabla^3u(\eta,\eta,\eta) + H\nabla^2u(\eta,\eta) + \oDelta\eta u - 2\odelta\left(A(\onabla u)\right) + \lp\onabla H,\onabla u\rp \\
   & \quad - \left(\oJ + nP(\eta,\eta) + \frac{1}{2n}H^2 + \frac{2n-1}{2(n-1)}\lv A_0\rv^2\right)\eta u. 
 \end{align*}
\end{lem}

\begin{remark}
 The tensor~\eqref{eqn:fialkow} is manifestly conformally invariant.  It is often referred to as the \emph{Fialkow tensor}~\cite{Vyatkin2013}.
\end{remark}

\subsection{Poincar\'e--Einstein manifolds}

A \emph{Poincar\'e--Einstein manifold} is a triple $(X^{n+1},M^n,g_+)$ consisting of a complete Riemannian manifold $(X_0^{n+1},g_+)$ with $\Ric_{g_+}=-ng_+$ such that $X_0$ is (diffeomorphic to) the interior of a compact manifold $X$ with boundary $\partial X=M^n$ for which there exists a defining function $r\in C^\infty(X)$ for $M$ such that $r^2g_+$ extends to a $C^{n-1,\alpha}$-metric on $X$.  Here a \emph{defining function for $M$} is a nonnegative function $r\in C^\infty(X)$ such that $r^{-1}(0)=M$ and $dr\not=0$ along $M$.  Since $\Ric_{g_+}=-ng_+$, any defining function necessarily satisfies $\lv dr\rv_{r^2g_+}=1$ along $M$.  A \emph{local defining function for $M$} is a function defined in a neighborhood $U$ of the boundary $M$ for which these same properties hold in $U$.  Note that if $r$ is a defining function, then so too is $e^\sigma r$ for any $\sigma\in C^\infty(X)$, and hence only the conformal class $[r^2g_+\rv_{TM}]$ is well-defined on a Poincar\'e--Einstein manifold.  We call the conformal manifold $(M^n,[r^2g_+\rv_{TM}])$ the \emph{conformal boundary of $(X,M,g_+)$}.

While defining functions are not uniquely determined by the data of a Poincar\'e--Einstein manifold, there are particularly nice defining functions determined when a representative $h\in[r^2g_+\rv_{TM}]$ of the conformal boundary is chosen.  Given such a representative, there is a unique local defining function $r$ such that $\lv dr\rv_{r^2g_+}\equiv1$; see~\cite{GrahamLee1991}.  That is, there is a defining function $r$ for $M$ such that $\lv dr\rv_{r^2g_+}\equiv1$ in some neighborhood $U$ of $M$, and if $\hr$ is another defining function with the same property, then $r\rv_U=\hr\rv_U$.  Since we are only concerned with the asymptotic behavior of the compactified metric $r^2g_+$ near $M$, the ambiguity of $r$ away from the boundary is immaterial.

The benefit of a geodesic defining function $r$ for the representative $h\in[r^2g_+\rv_{TM}]$ is that the flow lines of $\nabla r$ yield a diffeomorphism $U\cong [0,\varepsilon)\times M$ of a neighborhood $U$ of $M$, and in this neighborhood the metric $g:=r^2g_+$ takes the form $g=dr^2+h_r$ for $h_r$ a one-parameter family of Riemannian metrics on $M$, regarded here as metrics on the respective level sets $\{r\}\times M$.  Indeed, the Taylor expansion of $h_r$ around $r=0$ involves only even powers of $r$ up to order $n-1$, and
\begin{equation}
 \label{eqn:hr_expansion}
 h_r = h + h_{(2)}r^2 + h_{(4)}r^4 + \dotsb 
   + \begin{cases}
      h_{(n-1)}r^{n-1} + O(r^n), & \text{if $n$ is odd}, \\
      h_{(n-2)}r^{n-2} + O(r^n\log r), & \text{if $n$ is even} .
     \end{cases}
\end{equation}
Moreover, the terms $h_{(\ell)}$, $\ell<n$, are locally determined by $h$.    For example,
\begin{align}
 \label{eqn:h2} h_{(2)} & = -\oP, \\
 \label{eqn:h4} \tr_h h_{(4)} & = \frac{1}{4}\lv\oP\rv^2 ;
\end{align}
see~\cite{FeffermanGraham2012}.

Graham and Zworski~\cite{GrahamZworski2003} defined the fractional GJMS operators via scattering theory for the Laplacian of a Poincar\'e--Einstein manifold $(X^{n+1},M^n,g_+)$.  Specifically, given $\gamma\in\bigl(0,\frac{n}{2}\bigr)\setminus\bN$, set $s:=\frac{n}{2}+\gamma$ and suppose that $s(n-s)$ does not lie in the pure-point spectrum $\sigma_{pp}(-\Delta_{g_+})$ of $-\Delta_{g_+}$.  Then for any $f\in C^\infty(M)$, there is a unique solution of the Poisson equation
\begin{equation}
 \label{eqn:poisson_equation}
 -\Delta_{g_+}v - s(n-s)v = 0 
\end{equation}
such that
\begin{equation}
 \label{eqn:poisson_dirichlet}
 \lim_{r\to0}r^{s-n}v=f .
\end{equation}
Indeed, any solution of~\eqref{eqn:poisson_equation} must be of the form
\begin{equation}
 \label{eqn:poisson_asymptotics}
 v = Fr^{n-s} + Gr^s 
\end{equation}
for $F,G\in C^\infty(X)$, so our condition~\eqref{eqn:poisson_dirichlet} specifies that $F\rv_M=f$.  We denote this solution $\mP(s)(f)$ and define the \emph{scattering operator $S(s)$} by $S(s)f:=G\rv_M$, where $G$ is given by~\eqref{eqn:poisson_asymptotics} for $v=\mP(s)(f)$.  The \emph{fractional GJMS operator of order $2\gamma$} is
\begin{equation}
 \label{eqn:fractional_gjms_definition}
 P_{2\gamma} := d_\gamma S\left(\frac{n}{2}+\gamma\right), \qquad \text{where} \quad d_\gamma = 2^{2\gamma}\frac{\Gamma(\gamma)}{\Gamma(-\gamma)} . 
\end{equation}
Graham--Zworski showed that $P_{2\gamma}$ is a conformally covariant, formally self-adjoint pseudodifferential operator with principal symbol that of $(-\Delta)^\gamma$.

The Taylor expansion of $F$ (resp.\ of $G$) at $r=0$ can be obtained from the Taylor expansion~\eqref{eqn:hr_expansion} and $F\rv_M$ (resp.\ of $G\rv_M$) by formally solving the equation
\[ -\Delta_{g_+}v - s(n-s)v = O(r^\infty) \]
for $v=Fr^{n-s}$ (resp.\ $v=Gr^s$); see~\cite{GrahamZworski2003}.  It follows that
\begin{equation}
 \label{eqn:poisson_taylor_asymptotics}
 \begin{split}
 F & = f + f_{(2)}r^2 + f_{(4)}r^4 + \dotsb, \\
 G & = \cf + \cf_{(2)}r^2 + \cf_{(4)}r^4 + \dotsb
 \end{split}
\end{equation}
are even up to order $n$, the functions $f_{(2\ell)}$, $\ell\leq n$, are locally determined by $f$ and $h$, and the functions $\cf_{(2\ell)}$, $\ell\leq n$, are locally determined by $\cf:=S(s)f$ and $h$.  For example, $f_{(2)}=T_2(s)f$ and $f_{(4)}=T_4(s)f$, where
\begin{equation}
 \label{eqn:expansion-operators}
 \begin{split}
 T_2(s) & = -\frac{1}{2(2s-n-2)}L_2(n-s) , \\
 T_4(s) & = \frac{1}{8(2s-n-4)}\left(\frac{1}{2s-n-2}L_2(n-s+2)L_2(n-s) + L_4(n-s) \right)
 \end{split}
\end{equation}
and
\begin{align*}
 L_2(s) & = -\oDelta + s\oJ, \\
 L_4(s) & = \odelta\bigl(2\oP - \oJ\og\bigr)\od + \oJ\oDelta - s\lv\oP\rv^2 .
\end{align*}
Likewise $\cf_{(2)}=T_2(n-s)\cf$ (see, for example, \cite{Luo2018,Wang2016}).

A \emph{compactification of a Poincar\'e--Einstein manifold} is a compact Riemannian manifold $(X^{n+1},g)$ with boundary $M$ for which there is a Poincar\'e--Einstein manifold $(X_0,M,g_+)$ such that $X_0$ is diffeomorphic to the interior of $X$ and there is a defining function $r$ for $\partial X$ such that $g=r^2g_+$.  This notion is clearly conformally invariant, and allows us to relate our conformally invariant boundary operators to the fractional GJMS operators of Graham and Zworski.

Our strengthening of Theorem~\ref{thm:boundary_operators} and Theorem~\ref{thm:L2-trace} is made by imposing conformally invariant conditions on the boundary of a compact manifold.  It is well-known~\cite{FeffermanGraham2012} that these conditions are automatically satisfied by compactifications of Poincar\'e--Einstein manifolds.

\begin{lem}
 \label{lem:coronal}
 Let $(X^{n+1},g)$ be a compactification of a Poincar\'e--Einstein manifold $(X_0^{n+1},M^n,g_+)$.  Then, as sections of $S^2T^\ast M$,
 \begin{align}
  \label{eqn:compactification_umbilic} A_0 & \equiv 0 , \\
  \label{eqn:compactification_weyl} W(\eta,\cdot,\eta,\cdot) & \equiv 0, \\
  \label{eqn:compactification_cotton} C(\eta,\cdot,\cdot) & \equiv 0 . 
 \end{align}
 Moreover, these conditions are conformally invariant.
\end{lem}

\begin{proof}
 Equation~\eqref{eqn:compactification_umbilic} is equivalent to the condition that the boundary $M$ is umbilic.  It is well-known that the condition of umbilicity is conformally invariant and that the lack of a first-order term in the expansion~\eqref{eqn:hr_expansion} implies that $M$ is umbilic (cf.\ \cite{FeffermanGraham2012}).  Note that the Codazzi--Mainardi equation,
 \[ R(Y_1,Y_2,\eta,Y_3) = \frac{1}{n}\left( \lp\nabla H,Y_2\rp\lp Y_1,Y_3\rp - \lp\nabla H,Y_1\rp\lp Y_2,Y_3\rp \right) \]
 for sections $Y_1,Y_2,Y_3$ of $TM$, and Lemma~\ref{lem:gauss-codazzi} together imply that, since $M$ is umbilic,
 \begin{equation}
  \label{eqn:umbilic_consequence}
  W(\eta,\cdot,\cdot,\cdot) \equiv 0
 \end{equation}
 as a section of $T^\ast M\otimes \Lambda^2T^\ast M$.
 
 Since the Weyl tensor is conformally invariant and the outward-pointing unit normals $\eta$ and $\heta$ with respect to $g$ and $\hg:=e^{2\sigma}g$, respectively, are related by $\heta=e^{-\sigma}\eta$, it is clear that~\eqref{eqn:compactification_weyl} is a conformally invariant condition.  The fact that~\eqref{eqn:compactification_weyl} holds is well-known~\cite{FeffermanGraham2012} and follows readily from the expansion $g=dr^2+h_r$.
 
 The conformal transformation formula
 \[ \hC_{ijk} = C_{ijk} - W_{ijkl}\sigma^l \]
 and the consequence~\eqref{eqn:umbilic_consequence} of umbilicity together imply that~\eqref{eqn:compactification_umbilic}, \eqref{eqn:compactification_weyl} and~\eqref{eqn:compactification_cotton} are a conformally invariant system of equations.  Since $g_+$ is Einstein, the definition $C=dP$ of the Cotton tensor implies that $C^{g_+}(\partial_r,\cdot,\partial_r)\equiv0$ near $M$, and hence~\eqref{eqn:compactification_cotton} holds.
\end{proof}

Our computations of boundary operators require only compact manifolds with boundary which satisfy the conclusions of Lemma~\ref{lem:coronal}.  For this reason, it is useful to give such boundaries a name.

\begin{defn}
 A Riemannian manifold $(X^{n+1},g)$ has \emph{{\coronal} boundary} $M=\partial X$ if each of~\eqref{eqn:compactification_umbilic}, \eqref{eqn:compactification_weyl} and~\eqref{eqn:compactification_cotton} holds along $M$.
\end{defn}

It is well-known that a Riemannian manifold of dimension at least four with vanishing Weyl tensor also has vanishing Cotton and Bach tensors.  Similarly, a Riemannian manifold with {\coronal} boundary is such that certain additional components of the Cotton and Weyl tensors vanish (cf.\ \cite{FeffermanGraham2012}).

\begin{lem}
 \label{lem:coronal-corollary}
 Let $(X^{n+1},g)$, $n\geq3$, be a Riemannian manifold with {\coronal} boundary.  Then, as sections of the tensor-algebra bundle of $X$,
 \begin{align}
  \label{eqn:coronal-corollary-W} W(\eta,\cdot,\cdot,\cdot) & \equiv 0, \\
  \label{eqn:coronal-corollary-C} C(\eta,\cdot,\cdot) & \equiv 0, \\
  \label{eqn:coronal-corollary-B} B(\eta,\cdot) & \equiv 0 .
 \end{align}
\end{lem}

\begin{proof}
 In what follows, let $Y\in T_pM$, where $M=\partial X$, and let $\{E_j\}_{j=1}^n$ be an orthonormal basis of $T_pM$.  All computations are done at the point $p\in M$, and we extend $Y$ and $E_j$ to a neighborhood of $p$ in $X$ such that $(\nabla Y)_p=(\nabla E_j)_p=0$.
 
 First, \eqref{eqn:compactification_weyl} and the consequence~\eqref{eqn:umbilic_consequence} of umbilicity imply that~\eqref{eqn:coronal-corollary-W} holds.
 
 Second, recall that $(n-2)C(\eta,Y,\eta)=(\delta W)(\eta,\eta,Y)$.  From the symmetries of the Weyl tensor, we compute that
 \begin{multline*}
  (\delta W)(\eta,\eta,Y) \\ = \sum_{j=1}^n \left[ E_j\bigl(W(E_j,\eta,\eta,Y)\bigr) - W(E_j,\nabla_{E_j}\eta,\eta,Y) - W(E_j,\eta,\nabla_{E_j}\eta,Y) \right] .
 \end{multline*}
 Since $M$ is umbilic and $W(\eta,Y,E_j,\eta)=0$, we see that $(\delta W)(\eta,\eta,Y)=0$.  Thus~\eqref{eqn:coronal-corollary-C} holds.
 
 Third, it follows from~\eqref{eqn:coronal-corollary-W} that
 \[ B(\eta,\cdot) = \sum_{j=1}^n \nabla_{E_j} C(E_j,\eta,\cdot) . \]
 Computing as in the previous paragraph shows that~\eqref{eqn:coronal-corollary-B} holds.
\end{proof}

In order to make the computations necessary to prove Theorem~\ref{thm:dirichlet-to-neumann}, it is useful to compute a number of intrinsic and extensic scalar invariants of the boundary of a geodesic compactification of a Poincar\'e-Einstein manifold.

\begin{lem}
 \label{lem:geodesic_compactification_invariants}
 Let $(X^{n+1},g)$, $n\geq5$, be a geodesic compactification of a Poincar\'e--Einstein manifold.  Then, along $M$,
 \begin{align*}
  H & = 0, & P(\eta,\eta) & = 0, \\
  J & = \oJ, & \eta J & = 0, \\
  \nabla_\eta P(\eta,\eta) & = 0, & \Delta J & = \oDelta\oJ + \lv\oP\rv^2, \\
  \eta\Delta J & = 0 .
 \end{align*}
\end{lem}

\begin{proof}
 Since $g$ is a geodesic compactification of a Poincar\'e--Einstein manifold with $n\geq5$, in a neighborhood $U$ of $M$ it holds that $g=dr^2+h_r$ for $h_r$ a one-parameter family of Riemannian metrics on $M$ such that
 \begin{equation}
  \label{eqn:metric_expansion}
  h_r = h - r^2\oP + r^4h_{(4)} + r^5\ch + o(r^5),
 \end{equation}
 where $\tr_hh_{(4)}=\frac{1}{4}\lv\oP\rv^2$ and $\tr_h\ch=0$; see~\cite{FeffermanGraham2012}.  In particular, we see that the outward-pointing unit normal along $M:=\partial X$ is $\eta=-\partial_r$.  Since $\partial_rg\equiv0$ along $M$, we recover the well-known fact that $M$ is totally geodesic.
 
 Next, since $g_+=r^{-2}g$ and $P^{g_+}=-\frac{1}{2}g_+$, we deduce from the conformal transformation of the Schouten tensor that
 \[ P^g = -r^{-1}\nabla^2r , \]
 where the Hessian on the right-hand side is computed with respect to $g$.  Since $\nabla_{\partial_r}\partial_r=0$ in $U$, we see that $P(\partial_r,\cdot)\equiv0$.  In particular, $P(\eta,\eta)=0$ and $\nabla_\eta P(\eta,\eta)=0$.  Using~\eqref{eqn:metric_expansion}, we also compute that
 \[ P^g\rv_{TM} = \oP - 2r^2h_{(4)} - \frac{5}{2}r^3\ch + o(r^3) . \]
 Taking the trace with respect to $g$ yields
 \begin{equation}
  \label{eqn:J_expansion}
  J = \oJ + \frac{1}{2}r^2\lv\oP\rv^2 + o(r^3) .
 \end{equation}
 In particular, $J\rv_M=\oJ$ and $\eta J=0$.
 
 Finally, it follows from~\eqref{eqn:metric_expansion} that, near $M$,
 \[ \Delta = \oDelta + \partial_r^2 - r\oJ\partial_r + O(r^2) . \]
 Combining this with~\eqref{eqn:J_expansion} yields
 \[ \Delta J = \oDelta\oJ + \lv\oP\rv^2 + o(r) . \]
 In particular, $\Delta J\rv_M=\oDelta\oJ+\lv\oP\rv^2$ and $\eta\Delta J=0$.
\end{proof}

\subsection{Differential (boundary) operators}

There are five basic types of differential operators from which we build our boundary operators.  First are the exterior derivatives $d\colon C^\infty(X)\to\Omega^1(X)$ and $\od\colon C^\infty(M)\to\Omega^1(M)$.  Second are the divergences $\delta\colon\Omega^1(X)\to C^\infty(X)$ and $\odelta\colon\Omega^1(M)\to C^\infty(M)$.  Third is the outward-pointing normal, regarded as a derivation $\eta\colon C^\infty(X)\to C^\infty(M)$.  Fourth is the action of a symmetric $(0,2)$-tensor on one-forms as an endomorphism.  Fifth is the action of a smooth function as a multiplication operator.  Juxtapositions of such operators denote compositions; for example, $\delta Pd$ denotes the operator
\[ (\delta Pd)(u) := \nabla^i\left(P_i^j\nabla_ju\right) \]
for all $u\in C^\infty(X)$, while $\oDelta:=\odelta\od$ denotes the intrinsic Laplacian on $M$ with respect to the metric induced by $g$.  When necessary, we surround expressions for scalar invariants by parentheses to clarify how they act; e.g.\ as operators, $\Delta J$ denotes the operator $u\mapsto\Delta(Ju)$ while $(\Delta J)$ denotes the multiplication operator $u\mapsto u\Delta J$.

Our operators are all natural Riemannian operators (see~\cite{AtiyahBottPatodi1973} and~\cite[Subsection~2.4]{GoverWaldron2015} for precise definitions), and can be regarded as functions which map each metric $g$ in $X$ to a differential (boundary) operator on the space of smooth functions or one-forms on $X$ or $M$ which is polynomial in the metrics $g$ and $\og$, their inverses, their Levi-Civita connections $\nabla$ and $\onabla$, their Riemann curvature tensors, and the outward-pointing unit normal along $M$.  This allows us to discuss the homogeneity of such an operator: A natural Riemannian operator $L^g$ is \emph{homogeneous of degree $k$} if $L^{c^2g}=c^{k}L^g$ for all constants $c>0$.  For example, the Laplacians $\Delta$ and $\oDelta$ are homogeneous of degree $-2$, while the normal derivative $\eta$ is homogeneous of degree $-1$.

A homogeneous natural Riemannian operator $L\colon C^\infty(X)\to C^\infty(M)$ of degree $k$ is \emph{conformally covariant of weight $w\in\bR$} if
\begin{equation}
 \label{eqn:conformal_covariant}
 L^{e^{2\sigma}g} = e^{(w+k)\sigma} L^g e^{-w\sigma}
\end{equation}
for all metrics $g$ on $X$ and all $\sigma\in C^\infty(X)$.  Recall that the right-hand side of~\eqref{eqn:conformal_covariant} expresses the pre- and post-composition of $L$ with two multiplication operators.  Note that~\eqref{eqn:conformal_covariant} specifies that $L$ is \emph{conformally covariant of bidegree (w,w+k)}.

It is typically simpler to check whether a homogeneous natural Riemannian operator $L$ is infinitesimally conformally covariant.  If $L$ is homogeneous of degree $k$, then we define the infinitesimal conformal variation of $L$ in the direction of $\sigma\in C^\infty(X)$ acting on densities of weight $w\in\bR$ by
\begin{equation}
 \label{eqn:deltaw}
 \delta_w(L)_{g,\sigma} := \left.\frac{\partial}{\partial t}\right|_{t=0} e^{-(w+k)t\sigma}L^{e^{2t\sigma}g}e^{wt\sigma} .
\end{equation}
We say that $L$ is \emph{infinitesimally conformally covariant of weight $w\in\bR$} if $\delta_w(L)_{g,\sigma}=0$ for all $g$ and $\sigma$.  Branson~\cite{Branson1985} showed that $L$ is infinitesimally conformally covariant of weight $w$ if and only if it is conformally covariant of weight $w$.  When the metric $g$ and the function $\sigma$ are clear from context, we simply write $\delta_wL$ for the infinitesimal conformal variation~\eqref{eqn:deltaw}.

There are two observations that make it relatively simple to compute infinitesimal conformal variations.  The first are the well-known computations of the conformal variations of basic operators, some of which are collected in the following lemma for convenience.

\begin{lem}
 \label{lem:basic_conformal_variations}
 Let $(X^{n+1},g)$ be a compact Riemannian manifold with boundary.  Let $\sigma\in C^\infty(X)$ denote a direction in the space of conformal metrics and let $\alpha\in\Omega^1(X)$.  Then
 \begin{align*}
  \delta_w g & = 0, \\
  \delta_w P & = -\nabla^2\sigma, \\
  (\delta_w\nabla)(\alpha) & = (w-1)d\sigma\otimes\alpha - \alpha\otimes d\sigma + \lp d\sigma,\alpha\rp\,g , \\
  \delta_w H & = n(\eta\sigma) .
 \end{align*}
\end{lem}

\begin{proof}
 It is immediate that $\delta_wg=0$.  The remaining conclusions follow from the well-known conformal transformation formulas
 \begin{align*}
  \hP & = P - \nabla^2\sigma + d\sigma\otimes d\sigma - \frac{1}{2}\lv\nabla\sigma\rv^2 g, \\
  \hnabla\alpha & = -d\sigma\otimes\alpha - \alpha\otimes d\sigma + \lp d\sigma,\alpha\rp\,g, \\
  e^\sigma\hH & = H + n\eta\sigma
 \end{align*}
 for $\hg:=e^{2\sigma}g$.
\end{proof}

The second observation is that the infinitesimal conformal variation operator $\delta_w$ satisfies an analogue of the Leibniz rule.

\begin{lem}
 \label{lem:deltaw_leibniz}
 Let $(X^{n+1},g)$ be a compact Riemannian manifold with boundary.  Suppose that $K$ and $L$ are natural Riemannian operators which are homogeneous of degrees $k$ and $\ell$, respectively, and that the composition $KL$ makes sense.  Then $KL$ is homogeneous of degree $k+\ell$ and
 \[ \delta_w\left(KL\right) = K\circ\delta_w L + \left(\delta_{w+\ell}K\right)\circ L . \]
\end{lem}

\begin{proof}
 This follows from the identity
 \[ e^{-(w+k+\ell)\sigma}KLe^{w\sigma} = e^{-(w+k+\ell)\sigma}Ke^{(w+\ell)\sigma}e^{-(w+\ell)\sigma}Le^{w\sigma} . \qedhere \]
\end{proof}
We use Lemma~\ref{lem:deltaw_leibniz} repeatedly in Section~\ref{sec:invariant} to compute conformal variations.  Rather than give full details of those computations, where they are rather lengthy, we give three simple consequences of Lemma~\ref{lem:deltaw_leibniz} which illustrate its use.

First, it follows from Lemma~\ref{lem:deltaw_leibniz} and the definition $\Delta=\tr\nabla d$ of the Laplace operator that
\begin{equation}
 \label{eqn:conformal_linearization_laplace}
 \left(\delta_w\Delta\right)(u) = \tr\left((\delta_w\nabla)d + \nabla(\delta_wd)\right)(u) = (n+2w-1)\lp\nabla\sigma,\nabla u\rp + wu\Delta\sigma
\end{equation}
on any $(n+1)$-dimensional Riemannian manifold.

Second, using Lemma~\ref{lem:deltaw_leibniz} and~\eqref{eqn:conformal_linearization_laplace}, we see that
\[ (\delta_w\eta\Delta)(u) = (n+2w-1)\eta\lp\nabla u,\nabla\sigma\rp + w\eta\left(u\Delta\sigma\right) + (w-2)(\eta\sigma)\Delta u . \]
Expanding this yields
\begin{multline}
 \label{eqn:conformal_linearization_etalaplace}
 (\delta_w\eta\Delta)(u) = (n+2w-1)\nabla^2u(\eta,\nabla\sigma) + w(\Delta\sigma)\eta u \\ + (n+2w-1)\nabla^2\sigma(\eta,\nabla u) + (w-2)(\eta\sigma)\Delta u + wu\eta\Delta\sigma .
\end{multline}
In Section~\ref{sec:invariant}, we use Lemma~\ref{lem:gauss-codazzi} to rewrite~\eqref{eqn:conformal_linearization_etalaplace} in terms of intrinsic operators on the boundary.

Third, if $I$ is any natural scalar Riemannian invariant which is homogeneous of degree $k$, we may define its \emph{conformal linearization} by
\[ DI := \left.\frac{\partial}{\partial t}\right|_{t=0} e^{-kt\sigma}I^{e^{2t\sigma}g} , \]
where we again suppress the dependence on $\sigma$ in our notation.  It is clear that when $I$ is regarded as a multiplication operator, $\delta_wI = \delta_0I$ for all $w\in\bR$.  Moreover, $DI=(\delta_0I)(1)$.  Therefore we may use Lemma~\ref{lem:deltaw_leibniz} to compute conformal linearizations.  For example,
\[ D(\Delta J) = \left(\delta_{-2}\Delta\right)(J) + \Delta\left(DJ\right) . \]
Combined with Lemma~\ref{lem:basic_conformal_variations} we obtain
\begin{equation}
 \label{eqn:conformal_linearization_DeltaJ}
 D(\Delta J) = (n-5)\lp\nabla\sigma,\nabla J\rp - 2J\Delta\sigma - \Delta^2\sigma . 
\end{equation}
\section{Covariant operators associated to $L_6$}
\label{sec:invariant}

Finding covariant operators associated to $L_6$ involves two tasks.  First, we must find conformally covariant boundary operators of normal order $j\in\{0,\dotsc,5\}$ and of bidegree $\bigl(-\frac{n-5}{2},-\frac{n-5+2j}{2}\bigr)$.  Second, we must check that the associated bilinear form is symmetric.

Explicit formulas for conformally covariant operators of normal order $j\leq3$ and of bidegree $(w,w-j)$ were given by the first-named author~\cite{Case2015b}.  Specializing to the case $w=-\frac{n-5}{2}$ yields the following result.

\begin{prop}
 \label{prop:B0-3}
 Let $(X^{n+1},g)$ be a Riemannian manifold with umbilic boundary $M$ and define $B_j^5\colon C^\infty(X)\to C^\infty(M)$, $j\in\{0,1,2,3\}$, by
 \begin{align*}
  B_0^5(u) & := u, \\
  B_1^5(u) & := \eta u + \frac{n-5}{2}T_1^5u, \\
  B_2^5(u) & := \Delta u - \frac{4}{3}\oDelta u - \frac{4}{n}H\eta u + \frac{n-5}{2}T_2^5u, \\
  B_3^5(u) & := \eta\Delta u - 4\oDelta\eta u + \frac{n-9}{2n}H\nabla^2u(\eta,\eta) - \frac{3n-19}{2n}H\oDelta u - \frac{4(n-4)}{n}\lp\onabla H,\onabla u\rp \\
   & \quad + S_2^5\eta u + \frac{n-5}{2}T_3^5u
 \end{align*}
 where
 \[ S_2^5 := \frac{3n-7}{2}\oJ - \frac{n-13}{2}P(\eta,\eta) + \frac{3n^2-19n+36}{4n^2}H^2 \]
 and
 \begin{align*}
  T_1^5 & := \frac{1}{n}H, \\
  T_2^5 & = \frac{1}{3}\oJ - P(\eta,\eta) + \frac{n-4}{2n^2}H^2, \\
  T_3^5 & = -\eta J - \frac{4}{n}\oDelta H - \frac{n-9}{2n}HP(\eta,\eta) + \frac{3n-11}{2n}H\oJ + \frac{n^2-5n+12}{4n^3}H^3. 
 \end{align*}
 Then $B_j^5$ is conformally covariant of bidegree $\bigl(-\frac{n-5}{2},-\frac{n-5+2j}{2}\bigr)$.  Moreover, in the critical case $n=5$, it holds that
 \begin{equation}
  \label{eqn:critical-03}
  e^{j\sigma}\hT_j^5 = T_j^5 + B_j^5(\sigma)
 \end{equation}
 for all $1\leq j\leq3$ and all $\sigma\in C^\infty(X)$, where $\hT_j^5$ is defined in terms of $\hg:=e^{2\sigma}g$.
\end{prop}

\begin{proof}
 That $B_0^5$ and $B_1^5$ are conformally covariant of bidegree $\bigl(-\frac{n-5}{2},-\frac{n-5}{2}\bigr)$ and $\bigl(-\frac{n-5}{2},-\frac{n-3}{2}\bigr)$, respectively, follows from~\cite[Proposition~3.1]{Case2015b}.  That $B_2^5$ is conformally covariant of bidegree $\bigl(-\frac{n-5}{2},-\frac{n-1}{2}\bigr)$ follows from~\cite[Proposition~3.3]{Case2015b}.  That $B_3^5$ is conformally covariant of bidegree $\bigl(-\frac{n-5}{2},-\frac{n+1}{2}\bigr)$ follows from~\cite[Proposition~3.8]{Case2015b}.  The final claim~\eqref{eqn:critical-03} follows as in~\cite[Proposition~6.1]{Case2015b}.
\end{proof}

We now proceed to derive formulas for the boundary operators of normal order $4$ and $5$ associated to $L_6$.  These formulas were derived by computing conformal variations of various boundary operators.  We expect that these operators agree with the operators constructed by Gover and Peterson~\cite{GoverPeterson2018}.

The boundary operator of normal order $4$ associated to $L_6$ is as follows:

\begin{prop}
 \label{prop:B4}
 Let $(X^{n+1},g)$ be a Riemannian manifold with {\coronal} boundary $M$ and define $B_4^5\colon C^\infty(X)\to C^\infty(M)$ by
 \begin{align*}
	 B_4^5(u) & := -\Delta^2u - 4\oDelta\Delta u + 8\oDelta^2u + \frac{4}{n}H\eta\Delta u + \frac{16}{n}H\oDelta\eta u \\
   & \quad + \left((3n-5)\oJ + (n-11)P(\eta,\eta) - \frac{n^2-5n+18}{2n^2}H^2\right)\nabla^2u(\eta,\eta) \\
   & \quad - \left( 3(n-3)\oJ - (3n-13)P(\eta,\eta) + \frac{3n^2-23n+66}{2n^2}H^2 \right)\oDelta u \\
   & \quad + 8\odelta\left(\oP(\onabla u)\right) + \frac{48}{n}\lp\onabla H,\onabla\eta u\rp \\
   & \quad - \left\lp (3n-11)\onabla\oJ - (5n-29)\onabla P(\eta,\eta) + \frac{5n^2-53n+128}{2n^2}\onabla H^2, \onabla u\right\rp \\
   & \quad + S_3^5\eta u + \frac{n-5}{2}T_4^5u ,
 \end{align*}
 where
 \begin{multline*}
  S_3^5 := (n-9)\eta J + \frac{16}{n}\oDelta H \\ + \frac{3n^2-15n+10}{n}H\oJ + \frac{n^2-5n+26}{n}HP(\eta,\eta) - \frac{n^3-7n^2+12n-24}{2n^3}H^3
 \end{multline*}
 and
 \begin{align*}
  T_4^5 & := \Delta J - 4\oDelta\oJ + 4\oDelta P(\eta,\eta) - \frac{4(n-4)}{n^2}H\oDelta H - \frac{4}{n}H\eta J - 3(n-1)\oJ P(\eta,\eta) \\
   & \quad + \frac{n^2-3n+18}{2n^2}H^2P(\eta,\eta) + \frac{3n^2-13n+2}{2n^2}H^2\oJ - \frac{4(n-6)}{n^2}\lv\onabla H\rv^2 - 4\lv\oP\rv^2 \\
   & \quad + \frac{3(n-1)}{2}\oJ^2 - \frac{n-9}{2}P(\eta,\eta)^2 - \frac{n^3-5n^2+4n-24}{8n^4}H^4 .
 \end{align*}
 Then $B_4^5$ is conformally covariant of bidegree $\bigl(-\frac{n-5}{2},-\frac{n+3}{2}\bigr)$.  Moreover, in the critical case $n=5$, it holds that
 \begin{equation}
  \label{eqn:critical-4}
  e^{4\sigma}\hT_4^5 = T_4^5 + B_4^5(\sigma)
 \end{equation}
 for all $\sigma\in C^\infty(X)$, where $\hT_4^5$ is defined in terms of $\hg:=e^{2\sigma}g$.
\end{prop}

\begin{proof}
 Note that Lemma~\ref{lem:deltaw_leibniz} and~\eqref{eqn:conformal_linearization_laplace} imply that
 \begin{equation}
  \label{eqn:conformal_linearization_pre_bilaplace}
  \delta_{-\frac{n-5}{2}}(\Delta^2)(u) = -\frac{n-1}{2}(\Delta u)(\Delta\sigma) + 4\Delta\lp\nabla u,\nabla\sigma\rp - \frac{n-5}{2}\Delta\left(u\Delta\sigma\right) .
 \end{equation}
 Expanding the last two terms using the Bochner formula and product rule yields
 \begin{equation}
  \label{eqn:conformal_linearization_bilaplace}
  \begin{split}
   \delta_{-\frac{n-5}{2}}(\Delta^2)(u) & = 4\lp\nabla\sigma,\nabla\Delta u\rp - (n-3)(\Delta u)(\Delta\sigma) + 8\lp\nabla^2u,\nabla^2\sigma\rp \\
    & \quad + 8(n-1)P(\nabla u,\nabla\sigma) + 8J\lp\nabla u,\nabla\sigma\rp \\
    & \quad - (n-9)\lp\nabla u,\nabla\Delta\sigma\rp - \frac{n-5}{2}u\Delta^2\sigma .
  \end{split}
 \end{equation}
 We rewrite the first summand on the right-hand side using the identity
 \begin{equation}
  \label{eqn:B45-rewrite-lot}
  4\lp\nabla\sigma,\nabla\Delta u\rp = 4(\eta\sigma)\eta\Delta u + 4\lp\onabla\Delta u,\onabla\sigma\rp .
 \end{equation}
 Similar to~\eqref{eqn:conformal_linearization_bilaplace}, it holds that
 \begin{equation}
  \label{eqn:conformal_linearization_olaplacelaplace}
  \begin{split}
   \delta_{-\frac{n-5}{2}}(\oDelta\Delta)(u) & = 4\lp\onabla\oDelta u,\onabla\sigma\rp - \lp\onabla\Delta u,\onabla\sigma\rp + 4(\eta\sigma)\oDelta\eta u + 8\lp\onabla^2u,\onabla^2\sigma\rp \\
    & \quad + 8\lp\onabla\eta u,\onabla\eta\sigma\rp - \frac{n-5}{2}(\oDelta u)\Delta\sigma - \frac{n-1}{2}(\Delta u)\oDelta\sigma \\
    & \quad + 4\lp\onabla u,\onabla\oDelta\sigma\rp - (n-5)\lp\onabla u,\onabla\Delta\sigma\rp + 4(\eta u)\oDelta\eta\sigma \\
    & \quad + 8(n-2)\oP(\onabla u,\onabla\sigma) + 8\oJ\lp\onabla u,\onabla\sigma\rp - \frac{n-5}{2}u\oDelta\Delta\sigma
  \end{split}
 \end{equation}
 and also
 \begin{equation}
  \label{eqn:conformal_linearization_biolaplace}
  \begin{split}
   \delta_{-\frac{n-5}{2}}(\oDelta^2)(u) & = 2\lp\onabla\oDelta u,\onabla\sigma\rp + 6\lp\onabla^2u,\onabla^2\sigma\rp - (n-3)(\oDelta u)(\oDelta\sigma) \\
    & \quad + 6(n-2)\oP(\onabla u,\onabla\sigma) + 6\oJ\lp\onabla u,\onabla\sigma\rp \\
    & \quad - (n-8)\lp\onabla u,\onabla\oDelta\sigma\rp - \frac{n-5}{2}u\oDelta^2\sigma .
  \end{split}
 \end{equation}
 Combining~\eqref{eqn:conformal_linearization_bilaplace}, \eqref{eqn:B45-rewrite-lot},  \eqref{eqn:conformal_linearization_olaplacelaplace}, and~\eqref{eqn:conformal_linearization_biolaplace} yields
 \begin{align*}
  \MoveEqLeft[4] \delta_{-\frac{n-5}{2}}\left(-\Delta^2 - 4\oDelta\Delta + 8\oDelta^2\right)(u) \\
  & = -4(\eta\sigma)\eta\Delta u - 16(\eta\sigma)\oDelta\eta u - 8\lp\nabla^2u,\nabla^2\sigma\rp + 16\lp\onabla^2\sigma,\onabla^2u\rp \\
   & \quad + (n-3)(\Delta u)(\Delta\sigma) + 2(n-5)(\oDelta u)\Delta\sigma + 2(n-1)(\Delta u)\oDelta\sigma \\
   & \quad - 8(n-3)(\oDelta u)\oDelta\sigma - 32\lp\onabla\eta u,\onabla\eta\sigma\rp \\
   & \quad - 8(n-1)P(\nabla u,\nabla\sigma) - 8J\lp\nabla u,\nabla\sigma\rp + (n-9)(\eta\Delta\sigma)\eta u \\
   & \quad - 8(n-6)\lp\onabla u,\onabla\oDelta\sigma\rp + (5n-29)\lp\onabla u,\onabla\Delta\sigma\rp - 16(\eta u)\oDelta\eta\sigma  \\
   & \quad + 16(n-2)\oP(\nabla u,\onabla\sigma) + 16\oJ\lp\onabla u,\onabla\sigma\rp \\
   & \quad - \frac{n-5}{2}u\left(-\Delta^2\sigma - 4\oDelta\Delta\sigma + 8\oDelta^2\sigma\right) .
 \end{align*}
 The key point here is that there is no third-order term in $u$ which involves $\onabla\sigma$.  Combining the previous display with Lemma~\ref{lem:gauss-codazzi} and the conformal variations of the terms of order at most three from~\cite[Section~3]{Case2015b}, we see that the operator
 \[ \mathring{B}_4^5 := B_4^5 - \frac{n-5}{2}T_4^5, \]
 which has the property that $\mathring{B}_4^5(1)=0$, has conformal variation
 \begin{equation}
  \label{eqn:conformal_linearization_ring45}
  (\delta_{-\frac{n-5}{2}}\mathring{B}_4^5)(u) = -\frac{n-5}{2}u\mathring{B}_4^5(\sigma) .
 \end{equation}
 
 Next, applications of Lemma~\ref{lem:basic_conformal_variations} and Lemma~\ref{lem:deltaw_leibniz} imply that
 \begin{align}
  \label{eqn:conformal_linearization_oDeltaoPnn} D\left(\oDelta P(\eta,\eta)\right) & = -\oDelta\nabla^2\sigma(\eta,\eta) + (n-6)\lp\onabla P(\eta,\eta),\onabla\sigma\rp - 2P(\eta,\eta)\oDelta\sigma , \\
  \label{eqn:conformal_linearizaiton_oDeltaoJ} D(\oDelta\oJ) & = -\oDelta^2\sigma + (n-6)\lp\onabla\oJ,\onabla\sigma\rp - 2\oJ\oDelta\sigma .
 \end{align}
 Combining~\eqref{eqn:conformal_linearization_DeltaJ}, \eqref{eqn:conformal_linearization_oDeltaoPnn} and~\eqref{eqn:conformal_linearizaiton_oDeltaoJ} with the conformal linearizations from~\cite[Section~3]{Case2015b} yields
 \begin{equation}
  \label{eqn:conformal_linearization_T45}
  D(T_4^5) = \mathring{B}_4^5(\sigma) .
 \end{equation}
 It follows from~\eqref{eqn:conformal_linearization_ring45}, \eqref{eqn:conformal_linearization_T45} and the observation $\delta_wI=DI$ for any scalar invariant $I$ that
 \[ \delta_{-\frac{n-5}{2}}B_4^5 = 0 . \]
 In particular, $B_4^5$ is conformally covariant of bidegree $\bigl(-\frac{n-5}{2},-\frac{n+3}{2}\bigr)$, as claimed.
 
 Finally, if $n=5$, then $B_4^5=\mathring{B}_4^5$.  It then follows from~\eqref{eqn:conformal_linearization_ring45} and~\eqref{eqn:conformal_linearization_T45} that~\eqref{eqn:critical-4} holds.
\end{proof}

The boundary operator of normal order $5$ associated to $L_6$ is as follows:

\begin{prop}
 \label{prop:B5}
 Let $(X^{n+1},g)$ be a Riemannian manifold with {\coronal} boundary $M$ and define $B_5^5\colon C^\infty(X)\to C^\infty(M)$ by
 \begin{align*} 
   B_5^5 u &= \eta\Delta^2u + \frac{4}{3}\oDelta\eta\Delta u + \frac{8}{3}\oDelta^2\eta u + \frac{n+3}{2n}H\Delta^2u + \frac{2(n-9)}{3n} H\oDelta\nabla^2u(\eta,\eta) \\
    & \quad - \frac{4}{n}H\nabla^2(\Delta u)(\eta,\eta) + \frac{2(3n-11)}{3n}H\oDelta^2u \\
    & \quad - \left(\frac{5n-7}{3}\oJ + (n-7)P(\eta,\eta) - \frac{n^2-9n+10}{2n^2}H^2 \right)\eta\Delta u\\
    & \quad - \left( \frac{2(5n-9)}{3}\oJ + \frac{2(n-13)}{3}P(\eta,\eta) - \frac{3n^2-19n+12}{3n^2}H^2 \right)\oDelta\eta u \\
    & \quad + 8\eta\lp P,\nabla^2u\rp + 16\odelta\left(\oP(\onabla\eta u)\right) + \frac{4(n-12)}{3n}\lp\onabla H,\onabla\nabla^2u(\eta,\eta)\rp \\
    & \quad + \frac{4(5n-28)}{3n}\lp\onabla H,\onabla\oDelta u\rp + R_{1,3}^5\nabla^2u(\eta,\eta) + \frac{4(3n-7)}{n}H\odelta\left(\oP(\onabla u)\right) \\
    & \quad + \frac{8(2n-14)}{3n}\lp\onabla^2H,\onabla^2u\rp + R_{2,3}^5\oDelta u \\
    & \quad - \left\lp \onabla\left( \frac{15n-47}{3}\oJ + \frac{7n-79}{3}P(\eta,\eta) - \frac{15n^2-139n+168}{6n^2}H^2\right), \onabla\eta u \right\rp \\
    & \quad + \lp\sigma_4^5,\onabla u\rp + S_4^5\eta u + \frac{n-5}{2}T_5^5u ,
 \end{align*}
 where
 \begin{align*}
  R_{1,3}^5 & := -2(n-6)\eta J + \frac{2(n-9)}{3n}\oDelta H - \frac{5n^2-28n+15}{6n}H\oJ \\
   & \quad - \frac{n^2-16n+55}{2n}HP(\eta,\eta) + \frac{n^3-6n^2+11n-30}{4n^3}H^3 , \\
  R_{2,3}^5 & := -\frac{5n-19}{3}\eta J + \frac{2(5n-21)}{3n}\oDelta H - \frac{5n^2-20n+7}{2n}H\oJ \\
   & \quad - \frac{5(n-3)(n-5)}{6n}HP(\eta,\eta) + \frac{5n^3 - 26n^2 + 23n + 6}{12n^3}H^3 ,
 \end{align*}
 and
 \begin{align*}
  S_{4}^5 & := -\frac{n-5}{2}(\Delta J) - \frac{2(3n-11)}{3}\oDelta\oJ - (n-9)\nabla^2J(\eta,\eta) - \frac{2(n-13)}{3}\oDelta\left(P(\eta,\eta)\right) \\
   & \quad - \frac{16}{n}H\nabla_\eta P(\eta,\eta) + \frac{6n^2-38n+72}{3n^2}H\oDelta H + \frac{6n^2-62n+180}{3n^2}\lv\onabla H\rv^2 \\
   & \quad - \frac{3n^2-20n+13}{2n}H\eta J - \frac{3n^3-24n^2+103n-130}{4n^2}H^2P(\eta,\eta) \\
   & \quad - \frac{15n^3-68n^2-5n+42}{12n^2}H^2\oJ + \frac{5n^2-54n+49}{6}\oJ P(\eta,\eta) \\
   & \quad + \frac{5n^4-26n^3+17n^2-84n+120}{16n^4}H^4 + \frac{15n^2-50n-29}{12}\oJ^2 \\
   & \quad + \frac{n^2-22n+149}{4}P(\eta,\eta)^2 - 2(3n-11)\lv\oP\rv^2, \\
  \sigma_4^5 & := \frac{16(n-6)}{3n}\onabla\oDelta H + \frac{16n^2-96n-64}{3n}\oP(\onabla H) - \frac{7n-47}{3}\onabla\eta J \\
   & \quad - \frac{15n^2-70n+119}{6n}H\onabla\oJ - \frac{2(5n^2-45n+92)}{3n}\oJ\onabla H \\
   & \quad - \frac{2n^2-34n+168}{3n}P(\eta,\eta)\onabla H - \frac{7n^2-86n+303}{6n}H\onabla P(\eta,\eta) \\
   & \quad + \frac{3n^3-32n^2+117n-144}{6n^3}\onabla H^3
 \end{align*}
 and
 \begin{align*}
  T_5^5 & := -\eta\Delta J - \frac{4}{3}\oDelta\eta J + \frac{8}{3n}\oDelta^2H + \frac{4}{n}H\nabla^2J(\eta,\eta) - \frac{n+3}{2n}H\Delta J - \frac{2(3n-7)}{3n}H\oDelta\oJ \\
   & \quad - \frac{2(n-9)}{3n}H\oDelta P(\eta,\eta) - \frac{4(n-12)}{3n}\lp\onabla H,\onabla P(\eta,\eta)\rp - \frac{4(3n-16)}{3n}\lp\onabla H,\onabla\oJ\rp \\
   & \quad + \frac{5n-1}{3}\oJ\eta J + (n-5)P(\eta,\eta)\eta J - \frac{8}{n^2}H^2\nabla_\eta P(\eta,\eta) - 4\eta\lv P\rv^2 \\
   & \quad - \frac{n^2-7n-6}{2n^2}H^2\eta J - \frac{2(n-9)}{3n}P(\eta,\eta)\oDelta H + \frac{n^2-5n+12}{n^3}H^2\oDelta H \\
   & \quad + \frac{16}{n}\lp\oP,\onabla^2H\rp - \frac{10(n-1)}{3n}\oJ\oDelta H + \frac{15n^2-10n-37}{12n}H\oJ^2 \\
   & \quad + \frac{(n-5)(n-9)}{4n}HP(\eta,\eta)^2 - \frac{6(n-1)}{n}H\lv\oP\rv^2 + \frac{(n-5)(5n+3)}{6n}H\oJ P(\eta,\eta) \\
   & \quad - \frac{n^3-4n^2+33n-30}{4n^3}H^3P(\eta,\eta) + \frac{2(n-2)(n-7)}{n^3}H\lv\onabla H\rv^2 \\
   & \quad - \frac{5n^3-8n^2-19n-42}{12n^3}H^3\oJ + \frac{n^4-2n^3-3n^2-52n+24}{16n^5}H^5 .
 \end{align*}
 Then $B_5^5$ is conformally covariant of bidegree $\bigl(-\frac{n-5}{2},\frac{n+5}{2}\bigr)$.  Moreover, in the critical case $n=5$, it holds that
 \begin{equation}
  \label{eqn:critical-5}
  e^{5\sigma}\hT_5^5 = T_5^5 + B_5^5(\sigma)
 \end{equation}
 for all $\sigma\in C^\infty(X)$, where $\hT_5^5$ is defined in terms of $\hg:=e^{2\sigma}g$.
\end{prop}

\begin{proof}
 Using~\eqref{eqn:conformal_linearization_pre_bilaplace}, we deduce that
 \begin{equation}
  \label{eqn:conformal_linearization_etabilaplace}
  \begin{split}
   \delta_{-\frac{n-5}{2}}(\eta\Delta^2)(u) & = 4\lp\onabla\sigma,\onabla\eta\Delta u\rp + 4(\eta\sigma)\nabla^2(\Delta u)(\eta,\eta) - \frac{n+1}{2}(\eta\sigma)\Delta^2u \\
    & \quad - \frac{4}{n}H\lp\onabla\sigma,\onabla\Delta u\rp + 4\nabla^2\sigma(\eta,\nabla\Delta u) + 8\eta \delta\left(\nabla^2\sigma(\nabla u)\right) \\
    & \quad - \frac{n-5}{2}\eta\Delta\left(u\Delta\sigma\right) - \frac{n-1}{2}\eta\left((\Delta\sigma)\Delta u\right) + 4\eta\lp\nabla u,\nabla\Delta \sigma\rp .
  \end{split}
 \end{equation}
 Similarly, we compute that
 \begin{equation}
  \label{eqn:conformal_linearization_oDeltaetaDelta}
  \begin{split}
   \delta_{-\frac{n-5}{2}}(\oDelta\eta\Delta)(u) & = 4\lp\onabla\sigma,\onabla\oDelta\eta u\rp - 3\lp\onabla\sigma,\onabla\eta\Delta u\rp - \frac{n-1}{2}\oDelta\left((\eta\sigma)\Delta u\right) \\
    & \quad + 4\oDelta\left((\eta\sigma)\nabla^2u(\eta,\eta) + \nabla^2\sigma(\eta,\nabla u) - \frac{H}{n}\lp\onabla u,\onabla\sigma\rp \right) \\
    & \quad + 8\odelta\left(\onabla^2\sigma(\onabla\eta u)\right) - \frac{n+1}{2}(\oDelta\sigma)\eta\Delta u \\
    & \quad - \frac{n-5}{2}\oDelta\eta(u\Delta\sigma) - 4\lp\onabla\eta u,\onabla\oDelta\sigma\rp
  \end{split}
 \end{equation}
 and that
 \begin{equation}
  \label{eqn:conformal_linearization_biolaplaceeta}
  \begin{split}
   \delta_{-\frac{n-5}{2}}(\oDelta^2\eta u)(u) & = -2\lp\onabla\sigma,\onabla\oDelta\eta u\rp - \frac{n-5}{2}\oDelta^2(u\eta\sigma) + 2\odelta\left(\onabla^2\sigma(\onabla\eta u)\right) \\
   & \quad - (n-1)(\oDelta\sigma)\oDelta\eta u - (n-2)\lp\onabla\eta u,\onabla\oDelta\sigma\rp - \frac{n-3}{2}(\oDelta^2\sigma)\eta u .
  \end{split}
 \end{equation}
 Combining~\eqref{eqn:conformal_linearization_etabilaplace}, \eqref{eqn:conformal_linearization_oDeltaetaDelta}, and~\eqref{eqn:conformal_linearization_biolaplaceeta} yields
 \begin{align*}
  \MoveEqLeft[1] \delta_{-\frac{n-5}{2}}\left(\eta\Delta^2 + \frac{4}{3}\oDelta\eta\Delta + \frac{8}{3}\oDelta^2\eta\right)(u) \\
  & = 4(\eta\sigma)\nabla^2(\Delta u)(\eta,\eta) - \frac{n+1}{2}(\eta\sigma)\Delta^2u - \frac{2(n-1)}{3}\oDelta\left((\eta\sigma)\Delta u\right) \\
   & \quad + \frac{16}{3}\oDelta\left((\eta\sigma)\nabla^2u(\eta,\eta) + \nabla^2\sigma(\eta,\nabla u) - \frac{H}{n}\lp\onabla u,\onabla\sigma\rp \right) - \frac{4(n-5)}{3}\oDelta^2(u\eta\sigma) \\
   & \quad + 16\odelta\left(\onabla^2\sigma(\onabla\eta u)\right) - \frac{4}{n}H\lp\onabla\sigma,\onabla\Delta u\rp + 4\nabla^2\sigma(\eta,\nabla\Delta u) + 8\eta \delta\left(\nabla^2\sigma(\nabla u)\right) \\
   & \quad - \frac{n-5}{2}\eta\Delta\left(u\Delta\sigma\right) - \frac{n-1}{2}\eta\left((\Delta\sigma)\Delta u\right) - \frac{8(n-1)}{3}(\oDelta\sigma)\oDelta\eta u - \frac{2(n+1)}{3}(\oDelta\sigma)\eta\Delta u \\
   & \quad - \frac{2(n-5)}{3}\oDelta\eta(u\Delta\sigma) + 4\eta\lp\nabla u,\nabla\Delta \sigma\rp - \frac{8n}{3}\lp\onabla\eta u,\onabla\oDelta\sigma\rp - \frac{4(n-3)}{3}(\oDelta^2\sigma)\eta u .
 \end{align*}
 The key point here is that there is no fourth-order term in $u$ which involves $\onabla\sigma$.  Combining the previous display with Lemma~\ref{lem:gauss-codazzi}, \eqref{eqn:conformal_linearization_bilaplace}, \eqref{eqn:conformal_linearization_olaplacelaplace}, \eqref{eqn:conformal_linearization_biolaplace}, and the conformal variations from~\cite[Section~3]{Case2015b}, we see that the operator
 \[ \mathring{B}_5^5 := B_5^5 - \frac{n-5}{2}T_5^5, \]
 which has the property that $\mathring{B}_5^5(1)=0$, has conformal variation
 \begin{equation}
  \label{eqn:conformal_linearization_ring55}
  (\delta_{-\frac{n-5}{2}}\mathring{B}_5^5)(u) = -\frac{n-5}{2}u\mathring{B}_5^5(\sigma) .
 \end{equation}

 Next, applications of Lemma~\ref{lem:basic_conformal_variations} and Lemma~\ref{lem:deltaw_leibniz} imply that
 \begin{align}
  \label{eqn:conformal_linearization_etaDeltaJ} D(\eta\Delta J) & =  -\eta\Delta^2\sigma - 2J\eta\Delta\sigma + (n-5)\eta\lp\nabla J,\nabla\sigma\rp - 2(\eta J)\Delta\sigma - 4(\Delta J)\eta\sigma , \\
  \label{eqn:conformal_linearization_oDeltaetaJ} D(\oDelta\eta J) & = -\oDelta\eta\Delta\sigma - 2\oDelta(J\eta\sigma) + (n-8)\lp\onabla\eta J,\onabla\sigma\rp - 3(\eta J)\oDelta\sigma , \\
  \label{eqn:conformal_linearization_biolaplaceH} D(\oDelta^2H) & = n\oDelta^2\eta\sigma - H\oDelta^2\sigma - 
  (n-2)\lp\onabla H,\onabla\oDelta\sigma\rp \\
   \notag & \quad + 2(n-4)\odelta\left(\onabla^2\sigma(\onabla H)\right) - 4(\oDelta H)\oDelta\sigma + 2(n-6)\lp\onabla\oDelta H,\onabla\sigma\rp .
 \end{align}
 Combining~\eqref{eqn:conformal_linearization_etaDeltaJ}, \eqref{eqn:conformal_linearization_oDeltaetaJ}, and~\eqref{eqn:conformal_linearization_biolaplaceH} with the conformal linearizations~\eqref{eqn:conformal_linearization_DeltaJ}, \eqref{eqn:conformal_linearization_oDeltaoPnn}, \eqref{eqn:conformal_linearizaiton_oDeltaoJ} and those from~\cite[Section~3]{Case2015b} yields
 \begin{equation}
  \label{eqn:conformal_linearization_T55}
  D(T_5^5) = \mathring{B}_5^5(\sigma).
 \end{equation}
 It follows from~\eqref{eqn:conformal_linearization_ring55}, \eqref{eqn:conformal_linearization_T55}, and the observation $\delta_wI=DI$ for any scalar invariant $I$ that
 \[ \delta_{-\frac{n-5}{2}}B_5^5 = 0 . \]
 In particular, $B_5^5$ is conformally covariant of bidegree $\bigl(-\frac{n-5}{2},-\frac{n+5}{2}\bigr)$.
 
 Finally, if $n=5$, then $B_5^5=\mathring{B}_5^5$.  It then follows from~\eqref{eqn:conformal_linearization_ring55} and~\eqref{eqn:conformal_linearization_T55} that~\eqref{eqn:critical-5} holds. 
\end{proof}

It now remains to prove that the operators $B_j^5$, $0\leq j\leq 5$, are such that the bilinear form
\begin{equation}
 \label{eqn:L6-bilinear-form}
 \mQ_6(u,v) := \int_X u\,L_6v\,\dvol_g + \sum_{j=0}^2 \oint_M B_j^5(u)B_{5-j}^5(v)\,\dvol_{\og}
\end{equation}
is symmetric.  This follows from repeated applications of the Divergence Theorem.  The computation is significantly simplified by exploiting the fact that $\mQ_6$ is conformally covariant (see the proof of Proposition~\ref{prop:symmetric} below) and choosing a favorable metric in $X$ with which to compute (cf.\ \cite[Proposition~7.15]{GoverPeterson2018}).

\begin{lem}
 \label{lem:l6-normal-form}
 Let $(X^{n+1},g_0)$ be a compact Riemannian manifold with {\coronal} boundary $M$.  Then there is a metric $g=e^{2u}g_0$ such that
 \begin{equation}
  \label{eqn:l6-normal-form}
  \begin{cases}
   g\rv_{TM} = g_0\rv_{TM}, \\
   H = 0, \\
   P(\eta,\eta) = \frac{1}{3}\oJ, \\
   \eta J = 0, \\
   \Delta J = \frac{8}{3}\oDelta\oJ + 4\lv\oP\rv^2 - \frac{4n}{9}\oJ^2 , \\
   \eta\Delta J = -4\eta\lv P\rv^2 .
  \end{cases}
 \end{equation}
\end{lem}

\begin{proof}
 From Proposition~\ref{prop:B0-3}, Proposition~\ref{prop:B4}, and Proposition~\ref{prop:B5}, we see that if $n>5$, then $g:=u^{\frac{4}{n-5}}g_0$ satisfies~\eqref{eqn:l6-normal-form} if and only if $B_0^5(u)=1$ and $B_j^5(u)=0$ for $1\leq j\leq 5$.  This is readily arranged by taking $u\rv_M=1$ and recursively using the condition $B_j^5(u)=0$ to determine $\nabla_\eta^ju$.
 
 Likewise, if $n=5$, we see that $g:=e^{2u}g_0$ satisfies~\eqref{eqn:l6-normal-form} if and only if $B_0^5(u)=0$ and $T_j^5+B_j^5(u)=0$, which again is readily arranged.
\end{proof}

We now check that~\eqref{eqn:L6-bilinear-form} is symmetric.

\begin{prop}
 \label{prop:symmetric}
 Let $(X^{n+1},g)$ be a compact Riemannian manifold with {\coronal} boundary $M$.  Then~\eqref{eqn:L6-bilinear-form} is a conformally covariant symmetric bilinear form; i.e.\ for all $u,v,\sigma\in C^\infty(X)$, it holds that $\mQ_6(u,v)=\mQ_6(v,u)$ and
 \[ \mQ_6^{\hg}(u,v) = \mQ_6^g\left(e^{\frac{n-5}{2}\sigma}u,e^{\frac{n-5}{2}\sigma}v\right), \]
 where $\mQ_6^{\hg}$ is defined in terms of $\hg:=e^{2\sigma}g$.
\end{prop}

\begin{proof}
 The fact that $\mQ_6$ is conformally covariant follows from Proposition~\ref{prop:B0-3}, Proposition~\ref{prop:B4} and Proposition~\ref{prop:B5}.
 
 Since $\mQ_6$ is conformally covariant, $\mQ_6^g$ is symmetric if and only if $\mQ_6^{\hg}$ is symmetric for some $\hg\in[g]$.  We may thus assume without loss of generality that $g$ satisfies the conclusion of Lemma~\ref{lem:l6-normal-form}.
 
 A straightforward computation using Lemma~\ref{lem:gauss-codazzi} implies that if $g$ satisfies~\eqref{eqn:l6-normal-form}, then
 \begin{align*}
  B_0^5(u) & = u, \\
  B_1^5(u) & = \eta u, \\
  B_2^5(u) & = \Delta u - \frac{4}{3}\oDelta u, \\
  B_3^5(u) & = \eta\Delta u - 4\oDelta\eta u + \frac{4(n-1)}{3}\oJ\eta u, \\
  B_4^5(u) & = -\Delta^2u - 4\oDelta(\Delta u) + 8\oDelta^2u + 8\odelta\left(\oP(\onabla u)\right) \\
   & \quad + \frac{2(5n-13)}{3}\oJ\Delta u - \frac{8(2n-5)}{3}\oJ\oDelta u - \frac{4(n-1)}{3}\lp\onabla\oJ,\onabla u\rp , \\
  B_5^5(u) & = \eta\Delta^2u + \frac{4}{3}\oDelta\eta\Delta u + \frac{8}{3}\oDelta^2\eta u - \frac{2(3n-7)}{3}\oJ\eta\Delta u - \frac{16(2n-5)}{9}\oJ\oDelta\eta u \\
   & \quad + 16\odelta\left(\oP(\onabla\eta u)\right) + 8\eta\lp P,\nabla^2u\rp - (n-9)\nabla^2J(\eta,\eta)\eta u - \frac{4(13n-55)}{9}\lp\onabla\oJ,\onabla\eta u\rp \\
   & \quad - \frac{8(4n-19)}{9}(\oDelta\oJ)\eta u - 8(n-4)\lv\oP\rv^2\eta u + \frac{8(2n^2-10n+5)}{9}\oJ^2\eta u .
 \end{align*}
 Evaluating on $M$ using Lemma~\ref{lem:gauss-codazzi} and Lemma~\ref{lem:coronal-corollary}, we also see that
 \begin{align*}
  T_4(\eta, Y) & = 0, \\
  T_4(\eta, \eta) & = -\frac{8(n-5)}{3}\oDelta\oJ - 8(n-4)\lv\oP\rv^2 + \frac{8(2n^2-10n+5)}{9}\oJ^2, \\
  T_2(\eta, \nabla u) & = \frac{4(n-3)}{3}\oJ\eta u, \\
  \delta\left(T_2(\nabla u)\right) & = \frac{4(n-3)}{3}\oJ\Delta u + \frac{8}{3}\oJ\oDelta u - 8\lp\oP,\onabla^2u\rp + \frac{4(n-9)}{3}\lp\onabla\oJ,\onabla u\rp , \\
  \eta\delta\left(T_2(\nabla u)\right) & = \frac{4(n-1)}{3}\oJ\eta\Delta u - 8\eta\lp P,\nabla^2u\rp \\
   & \quad + (n-9)\nabla^2J(\eta,\eta)\eta u + \frac{4(n-9)}{3}\lp\onabla\oJ,\onabla\eta u\rp
 \end{align*}
 for all $Y\in TM$, where $T_2:=(n-1)Jg-8P$ and $T_4$ is given by~\eqref{eqn:L6-T4}.  Combining these formulas with the Divergence Theorem yields
 \[ \mQ_6(u,v) = \mF_I(u,v) + \mF_B(u,v), \]
 where $\mF_I$ is the interior bilinear form
 \begin{align*}
  \mF_I(u,v) & = \int \biggl[ \lp\nabla\Delta u,\nabla\Delta v\rp - T_2(\nabla u,\nabla\Delta v) - T_2(\nabla v,\nabla\Delta u) \\
   & \qquad\qquad - \frac{n-1}{2}J(\Delta u)(\Delta v) + T_4(\nabla u,\nabla v) + \frac{n-5}{2}Q_6uv \biggr]
 \end{align*}
 and $\mF_B$ is the boundary bilinear form
 \begin{align*}
  \mF_B(u,v) & = \oint \biggl[ -4\left((\Delta u)\oDelta\eta v + (\Delta v)\oDelta\eta u\right) + 8\left((\eta u)\oDelta^2v + (\eta v)\oDelta^2u\right) \\
   & \quad - 16\left(\oP(\onabla u,\onabla\eta v) + \oP(\onabla v,\onabla\eta u)\right) + \frac{8(n-2)}{3}\oJ\left((\Delta u)\eta v + (\Delta v)\eta u\right) \\
   & \quad + \frac{8(n-2)}{3}\oJ\left(\lp\onabla u,\onabla\eta v\rp + \lp\onabla v,\onabla\eta u\rp - (\oDelta u)\eta v - (\oDelta v)\eta u\right) \biggr] .
 \end{align*}
 It is clear by inspection that $\mF_I$ and $\mF_B$ are both symmetric.  Therefore $\mQ_6$ is symmetric.
\end{proof}

We conclude this section by discussing the generalized Dirichlet-to-Neumann operators associated to $L_6$ and the boundary operators $B_j^5$, $0\leq j\leq 5$.  These are well-defined on compact Riemannian manifolds with {\coronal} boundary for which the Dirichlet kernel of $L_6$,
\[ \ker_D L_6 := \left\{ u\in C^\infty(X) \suchthat B_0^5(u)=0, B_1^5(u)=0, B_2^5(u) = 0 \right\} , \]
is trivial.  Indeed:

\begin{prop}
 \label{prop:dirichlet-to-neumann-operators}
 Let $(X^{n+1},g)$ be a compact Riemannian manifold with {\coronal} boundary and for which $\ker_DL_6=\{0\}$.  Then for each $f,\phi,\psi\in C^\infty(M)$ there is a unique function $u_{f,\phi,\psi}\in C^\infty(X)$ which solves~\eqref{eqn:L6_extension}.  In particular, the operators $\mB_j^5\colon C^\infty(M)\to C^\infty(M)$, $j\in\{1,3,5\}$, given by
 \begin{align*}
  \mB_1^5(\psi) & := B_3^5\left(u_{0,0,\psi}\right), \\
  \mB_3^5(\phi) & := B_4^5\left(u_{0,\phi,0}\right), \\
  \mB_5^5(f) & := B_5^5\left(u_{f,0,0}\right),
 \end{align*}
 are well-defined formally self-adjoint pseudodifferential operators with principal symbol proportional to that of $(-\Delta)^{j/2}$ which depend only on $[g]$ and $h:=g\rv_{TM}$.  Moreover, $\mB_j^5$ are conformally covariant: For every $f,\sigma\in C^\infty(M)$, it holds that
 \[ \hmB_j^5\left(f\right) = e^{-\frac{n+j}{2}\sigma}\mB_j^5\left(e^{\frac{n-j}{2}\sigma}f\right), \]
 where $\hmB_j^5$ is defined with respect to any extension in $[g]$ of $\hh:=e^{2\sigma}h$.
\end{prop}

\begin{proof}
 Since the leading-order term of $B_j^5$, counted according to the number of transverse derivatives, is $\nabla_\eta^j$, it follows that $\bigl(L_6;B_0^5,B_1^5,B_2^5\bigr)$ satisfy the Lopatinksii--Shapiro conditions~\cite{AgmonDouglisNirenberg1959,Hormander2007,Lopatinskii1953,Shapiro1953} (cf.\ \cite{BransonGover2001,GoverPeterson2018}).  It then follows from the formal self-adjointness of $\bigl(L_6;B_0^5,B_1^5,B_2^5\bigr)$ and the assumption $\ker_DL_6=\{0\}$ that there is a unique smooth solution of~\eqref{eqn:L6_extension}.
 
 Since the sixth-order GJMS operator $L_6$ and the boundary operators $B_j^5$ are $O(TM)$-invariant and elliptic, the argument in~\cite[Theorem~8.4]{BransonGover2001} shows that $\mB_j^5$ are pseudodifferential operators with principal symbol proportional to that of $(-\Delta)^{j/2}$.  Since this construction is universal, the constant of proportionality is given by the conclusion of Theorem~\ref{thm:dirichlet-to-neumann}.
 
 Next, the definition of $\mB_5^5$ implies that
 \[ \oint_M f_1\,\mB_5^5 f_2\,\dvol_{\og} = \mQ_6\left(u_{f_1,0,0},u_{f_2,0,0}\right) . \]
 Proposition~\ref{prop:symmetric} implies that $\mB_5^5$ is formally self-adjoint.  The formal self-adjointness of $\mB_1^5$ and $\mB_3^5$ follows similarly.
 
 Finally, let $f,\phi,\psi\in C^\infty(M)$ and let $u:=u_{f,\phi,\psi}^g$ denote the extension~\eqref{eqn:L6_extension} of $(f,\phi,\psi)$ with respect to the background metric $g$.  Given $\Upsilon\in C^\infty(X)$, set $\sigma=\Upsilon\rv_M$.  Let $\hg:=e^{2\Upsilon}g$, $\hf:=e^{-\frac{n-5}{2}\sigma}f$, $\hphi:=e^{-\frac{n-3}{2}\sigma}\phi$, and $\hpsi:=e^{-\frac{n-1}{2}\sigma}\psi$.  It follows from Theorem~\ref{thm:boundary_operators} that $e^{-\frac{n-5}{2}\Upsilon}u=u_{\hf,\hphi,\hpsi}^{\hg}$ is the extension~\eqref{eqn:L6_extension} of $(\hf,\hphi,\hpsi)$ with respect to $\hg$.  Applying Theorem~\ref{thm:boundary_operators} again yields
 \begin{align*}
  \hB_3^5\left(u_{\hf,\hphi,\hpsi}^{\hg}\right) & = e^{-\frac{n+1}{2}\sigma}B_3^5\left(u_{f,\phi,\psi}\right), \\
  \hB_4^5\left(u_{\hf,\hphi,\hpsi}^{\hg}\right) & = e^{-\frac{n+3}{2}\sigma}B_4^5\left(u_{f,\phi,\psi}\right), \\
  \hB_5^5\left(u_{\hf,\hphi,\hpsi}^{\hg}\right) & = e^{-\frac{n+5}{2}\sigma}B_5^5\left(u_{f,\phi,\psi}\right) .
 \end{align*}
 Therefore $\mB_j^5$, $j\in\{1,3,5\}$, is independent of the choice of extension in $[g]$ of $h$ and that $\mB_j^5$ is conformally covariant of bidegree $\bigl(-\frac{n-j}{2},-\frac{n+j}{2}\bigr)$.
\end{proof}
\section{Fractional GJMS operators from $L_6$}
\label{sec:poincare}

We now study the relationship between the fractional GJMS operators and the boundary operators associated to the sixth-order GJMS operator.  To that end, it is useful to first give a simple formula for our boundary operators when computed with respect to a geodesic compactification of a Poincar\'e--Einstein manifold.

\begin{prop}
 \label{prop:geodesic_operators}
 Let $(X^{n+1},M^n,g_+)$, $n\geq5$, be a Poincar\'e--Einstein manifold, and let $r$ be a geodesic defining function for $M$.  With respect to the compactified metric $g:=r^2g_+$, the boundary operators $B_j^5$, $0\leq j\leq 5$, associated to $L_6$ are given by
 \begin{align*}
  B_0^5(u) & = u, \\
  B_1^5(u) & = -\partial_ru, \\
  B_2^5(u) & = \partial_r^2u + \frac{1}{3}\left(-\oDelta + \frac{n-5}{2}\oJ\right)u, \\
  B_3^5(u) & = -\partial_r^3u - 3\left(-\oDelta + \frac{n-3}{2}\oJ\right)\partial_r u, \\
  B_4^5(u) & = -\partial_r^4u + 6\left( -\oDelta + \frac{n-1}{2}\oJ\right)\partial_r^2u \\
   & \quad + 3\left(-\oDelta + \frac{n-1}{2}\oJ\right)\left(-\oDelta + \frac{n-5}{2}\oJ\right)u \\
   & \quad + 3\left( \odelta\left(2\oP-\oJ\og\right)\od + \oJ\oDelta - \frac{n-5}{2}\lv\oP\rv^2 \right)u , \\
  B_5^5(u) & = -\partial_r^5u + \frac{10}{3}\left(-\oDelta + \frac{n+1}{2}\oJ\right)\partial_r^3u \\
   & \quad - 5\left(-\oDelta + \frac{n+1}{2}\oJ\right)\left(-\oDelta + \frac{n-3}{2}\oJ\right)\partial_ru \\
   & \quad - 15\left( \odelta\left(2\oP-\oJ\og\right)\od + \oJ\oDelta - \frac{n-3}{2}\lv\oP\rv^2 \right)\partial_ru .
 \end{align*}
\end{prop}

\begin{proof}
 First note that $\eta=-\partial_r$ and $\nabla_\eta\eta=0$ in a neighborhood of $M$.  Therefore $\nabla^ku(\eta,\dotsc,\eta)=(-1)^k\partial_r^ku$ for all $k$.  Second, since $g=dr^2+h_r$ with $h_r$ as in~\eqref{eqn:hr_expansion}, it holds that
 \begin{equation}
  \label{eqn:pe_laplace_asymptotics}
  \Delta = \partial_r^2 + \oDelta - r\oJ\partial_r + \frac{r^2}{2}\left(\odelta\left(2\oP-\oJ\og\right)\od + \oJ\oDelta\right) - \frac{r^3}{2}\lv\oP\rv^2\partial_ru + O(r^4) .
 \end{equation}
 This implies that
 \begin{equation}
  \label{eqn:pe_bilaplace_asymptotics}
  \begin{split}
  \Delta^2 & = \partial_r^4u - 2(-\oDelta + \oJ)\partial_r^2u + \oDelta^2u + \odelta\left((2\oP-\oJ\og)(\onabla u)\right) + \oJ\oDelta u \\
   & \quad - r\bigl[ 2\oJ\partial_r^3u - 2\odelta\left((2\oP-\oJ\og)(\onabla\partial_ru)\right) + 2\lp\onabla\oJ,\onabla\partial_ru\rp \\
    & \qquad + \left(\oDelta\oJ + 3\lv\oP\rv^2 - \oJ^2\right)\partial_r u \bigr] + O(r^2) .
  \end{split}
 \end{equation}
 
 It follows from Lemma~\ref{lem:geodesic_compactification_invariants} and Proposition~\ref{prop:B0-3} that
 \begin{align*}
  B_0^5(u) & = u, \\
  B_1^5(u) & = \eta u, \\
  B_2^5(u) & = \Delta u - \frac{4}{3}\oDelta u + \frac{n-5}{6}\oJ u , \\
  B_3^5(u) & = \eta\Delta u - 4\oDelta\eta u + \frac{3n-7}{2}\oJ\eta u .
 \end{align*}
 The first two identities yield the claimed formulas for $B_0^5$ and $B_1^5$, respectively.  The final two identities and~\eqref{eqn:pe_laplace_asymptotics} together yield the claimed formulas for $B_2^5$ and $B_3^5$.
 
 Next, Lemma~\ref{lem:geodesic_compactification_invariants} and Proposition~\ref{prop:B4} imply that
 \begin{align*}
  B_4^5(u) & = -\Delta^2u - 4\oDelta\Delta u + 8\oDelta^2u + (3n-5)\oJ\partial_r^2u - 3(n-3)\oJ\oDelta u \\
   & \quad + 8\odelta\left(\oP(\onabla u)\right) - (3n-11)\lp\onabla\oJ,\onabla u\rp \\
   & \quad - \frac{3(n-5)}{2}\left( \oDelta\oJ + \lv\oP\rv^2 -\frac{n-1}{2}\oJ^2 \right)u. 
 \end{align*}
 Combining this display with~\eqref{eqn:pe_laplace_asymptotics} and~\eqref{eqn:pe_bilaplace_asymptotics} yields the claimed formula for $B_4^5$.
 
 Finally, observe that
 \[ \nabla^2u(X,Y) = \onabla^2u(X,Y) + r\oP(X,Y)\eta u + o(r) \]
 for all sections $X,Y$ of $TM$.  Since $P(\eta,\cdot)\equiv0$ and $\nabla_\eta P=0$, we conclude that
 \begin{equation}
  \label{eqn:geodesic_compactification_etaPHess}
  \eta\lp P,\nabla^2u\rp = \odelta\left(\oP(\onabla\eta u)\right) - \lp\onabla J,\onabla\eta u\rp - \lv\oP\rv^2\eta u .
 \end{equation}
 It follows from Lemma~\ref{lem:geodesic_compactification_invariants}, Proposition~\ref{prop:B5}, and~\eqref{eqn:geodesic_compactification_etaPHess} that
 \begin{align*}
  B_5^5(u) & = \eta\Delta^2u + \frac{4}{3}\oDelta\eta\Delta u + \frac{8}{3}\oDelta^2\eta u - \frac{5n-7}{3}\oJ\eta\Delta u - \frac{2(5n-9)}{3}\oJ\oDelta\eta u \\
   & \quad + 24\odelta\left(\oP(\onabla\eta u)\right) - \frac{15n-23}{3}\lp\onabla\oJ,\onabla\eta u\rp \\
   & \quad - \frac{15n-59}{6}(\oDelta\oJ)\eta u + \frac{15n^2-50n-29}{12}\oJ^2\eta u - \frac{15n-51}{2}\lv\oP\rv^2\eta u .
 \end{align*}
 Combining this display with~\eqref{eqn:pe_laplace_asymptotics} and~\eqref{eqn:pe_bilaplace_asymptotics} yields the claimed formula for $B_5^5$.
\end{proof}

Theorem~\ref{thm:dirichlet-to-neumann} asserts the desired relationship between the fractional GJMS operators and the boundary operators associated to $L_6$.  We recall this assertion below for the reader's convenience.

\begin{thm}
 \label{thm:dirichlet-to-neumann_restatement}
 Let $(X^{n+1},g)$ be a compactification of a Poincar\'e--Einstein manifold $(X_0,g_+)$ such that $\frac{n^2}{4}-\gamma^2\not\in\sigma_{pp}(-\Delta_{g_+})$ for $\gamma\in\{1/2,3/2,5/2\}$.  Suppose additionally that $u\in C^\infty(X)$ is such that $L_6u=0$.  Then
 \begin{equation}
  \label{eqn:dirichlet-to-neumann_restatement}
  \begin{split}
   B_5^5(u) & = \frac{8}{3}P_5\left(B_0^5(u)\right), \\
   B_4^5(u) & = 8P_3\left(B_1^5(u)\right), \\
   B_3^5(u) & = 3P_1\left(B_2^5(u)\right),
  \end{split}
 \end{equation}
 where $B_j^5$, $0\leq j\leq 5$, are the boundary operators of Theorem~\ref{thm:boundary_operators} and $P_{2\gamma}$, $\gamma\in\{1/2,3/2,5/2\}$ are the fractional GJMS operators of order $2\gamma$.
\end{thm}

\begin{proof}
 From the conformal covariance of the boundary operators $B_j^5$ and the fractional GJMS operators $P_{2\gamma}$, we see that we may assume that $g=r^2g_+$ for $r$ a geodesic defining function for $M$.
 
 Let $u\in\ker L_6$ and set $(f,\phi,\psi):=\left(B_0^5(u),B_1^5(u),B_2^5(u)\right)$.  Define
 \[ v_f:=\mP\left(\frac{n+5}{2}\right)(f), \quad v_\phi:=\mP\left(\frac{n+3}{2}\right)(\phi), \quad\text{and}\quad v_\psi:=\mP\left(\frac{n+1}{2}\right)(\psi) , \]
 where $\mP\bigl(\frac{n}{2}+s\bigr)$ denotes the solution of the Poisson equation~\eqref{eqn:poisson_equation} with prescribed Dirichlet data.  It follows from~\eqref{eqn:poisson_asymptotics} and~\eqref{eqn:poisson_taylor_asymptotics} that
 \begin{equation}
  \label{eqn:poincare_operators_expansion}
  \begin{split}
  v_f & = r^{\frac{n-5}{2}}\left(f + f_{(2)}r^2 + f_{(4)}r^4 + \cf r^5\right) + o\left(r^{\frac{n+5}{2}}\right) , \\
  v_\phi & = r^{\frac{n-3}{2}}\left(\phi + \phi_{(2)}r^2 + \cphi r^3 + \phi_{(4)}r^4\right) + o\left(r^{\frac{n+5}{2}}\right) , \\
  v_\psi & = r^{\frac{n-1}{2}}\left(\psi + \cpsi r + \psi_{(2)}r^2 + \cpsi_{(2)}r^3\right) + o\left(r^{\frac{n+5}{2}}\right) ,
  \end{split}
 \end{equation}
 where $\cf:=S\bigl(\frac{n+5}{2}\bigr)(f)$, $\cphi:=S\bigl(\frac{n+3}{2}\bigr)(\phi)$, and $\cpsi:=S\bigl(\frac{n+1}{2}\bigr)(\psi)$.
 
 Next, observe from Proposition~\ref{prop:geodesic_operators} that
 \begin{equation}
  \label{eqn:geodesic_operators_T}
  \begin{split}
   B_0^5(u) & = u, \\
   B_1^5(u) & = -\partial_ru, \\
   B_2^5(u) & = \partial_r^2u - 2T_2\left(\frac{n+5}{2}\right)u , \\
   B_3^5(u) & = -\partial_r^3u - 6T_2\left(\frac{n+3}{2}\right)B_1^5(u) , \\
   B_4^5(u) & = -\partial_r^4u + 12T_2\left(\frac{n+1}{2}\right)B_2^5(u) + 24T_4\left(\frac{n+5}{2}\right)u , \\
   B_5^5(u) & = -\partial_r^5u - 20T_2\left(\frac{n-1}{2}\right)B_3^5(u) - 120T_4\left(\frac{n+3}{2}\right)B_1^5(u) ,
  \end{split}
 \end{equation}
 where $T_2$ and $T_4$ are given by~\eqref{eqn:expansion-operators}.
 
 Now, since $g_+$ is Einstein with $\Ric_{g_+}=-ng_+$, the operator $L_6^{g_+}$ factors as
 \[ L_g^{g_+} = \left(-\Delta - \frac{n^2-25}{4}\right)\left(-\Delta - \frac{n^2-9}{4}\right)\left(-\Delta - \frac{n^2-1}{4}\right) ; \]
 see~\cite{FeffermanGraham2012,Gover2006q}.  In particular, $L_6^{g_+}(V)=0$, where $V:=v_f-v_\phi+\frac{1}{2}v_\psi$.  By conformal covariance, $L_6^g\bigl(v\bigr)=0$ for $v:=r^{-\frac{n-5}{2}}V$.  On the other hand, it follows from~\eqref{eqn:poincare_operators_expansion} that
 \begin{multline}
  \label{eqn:v_expansion_6}
  v = f - \phi r + \frac{1}{2}(\psi + 2f_{(2)})r^2 + \frac{1}{2}(\cpsi r^3 - 2\phi_{(2)})r^3 \\ + \frac{1}{2}(\psi_{(2)} - 2\cphi + 2f_{(4)})r^4 + \frac{1}{2}(\cpsi_{(2)} - 2\phi_{(4)} + 2\cf)r^5 + o(r^5) .
 \end{multline}
 Combining~\eqref{eqn:geodesic_operators_T} and~\eqref{eqn:v_expansion_6} yields $\bigl(B_0^5(v),B_1^5(v),B_2^5(v)\bigr) = (f,\phi,\psi)$.  Our assumption on $\sigma_{pp}(-\Delta_{g_+})$ implies that $\ker_DL_6=\{0\}$.  Hence solutions to the Dirichlet problem for $(L_6;B_0^5,B_1^5,B_2^5)$ are unique, and in particular $u=v$.  Moreover, \eqref{eqn:geodesic_operators_T} and~\eqref{eqn:v_expansion_6} also yield
 \begin{align*}
  B_3^5(u) & = -3\cpsi = -3S\left(\frac{n+1}{2}\right)\left(B_2^5(u)\right) , \\
  B_4^5(u) & = 24\cphi = 24S\left(\frac{n+3}{2}\right)\left(B_1^5(u)\right) , \\
  B_5^5(u) & = -120\cf = -120S\left(\frac{n+5}{2}\right)\left(B_0^5(u)\right) .
 \end{align*}
 The final conclusion follows from the definition~\eqref{eqn:fractional_gjms_definition} of the fractional GJMS operators.
\end{proof}

\section{Sobolev traces}
\label{sec:traces}

In this section we discuss a number of sharp Sobolev trace inequalities involving the $W^{3,2}(X)$-seminorm of a function when $\dim X>6$.  First we consider a norm inequality which establishes the trace embedding
\[ \Tr \colon W^{3,2}(X) \hookrightarrow H^5(\partial X) \oplus H^3(\partial X) \oplus H^1(\partial X) \]
provided the Dirichlet spectrum of $L_6$ is positive.  Indeed, we show that the energy functional $\mE_6$ is bounded below within the class of functions with prescribed Dirichlet data $\left(B_0^5(u),B_1^5(u),B_2^5(u)\right)$ by the energy of the unique such function with $L_6(u)=0$.  Our spectral assumption can be stated in terms of the first Dirichlet eigenvalue of $L_6$.

\begin{defn}
 Let $(X^{n+1},g)$, $n\geq5$, be a compact Riemannian manifold with {\coronal} boundary.  The first Dirichlet eigenvalue of $L_6$ is
 \[ \lambda_{1,D}(L_6) := \inf \left\{ \mE_6(u) \suchthat u \in \ker B_0^5\cap \ker B_1^5 \cap \ker B_2^5, \int_X u^2\,\dvol_{g} = 1 \right\} . \]
\end{defn}

Equivalently, $\lambda_{1,D}(L_6)$ is the infimum of the Rayleigh quotient $\mE_6(u)/\lV u\rV_{L^2(X)}^2$ over all functions with $B_0^5(u)=B_1^5(u)=B_2^5(u)=0$.  This latter characterization of $\lambda_{1,D}(L_6)$ implies that the assumption $\lambda_{1,D}(L_6)>0$ is conformally invariant.  It also allows us to minimize $\mE_6$ within the class of functions with prescribed Dirichlet data. 

\begin{thm}
 \label{thm:genl-L2-trace}
 Let $(X^{n+1},g)$, $n\geq5$, be a compact Riemannian manifold with {\coronal} boundary and $\lambda_{1,D}(L_6)>0$.  Given any $f,\phi,\psi\in C^\infty(M)$, it holds that
 \[ \mE_6(u) \geq \oint_M \left[ f\,\mB_5^5(f) + \phi\,\mB_3^5(\phi) + \psi\,\mB_1^5(\psi) \right]\,\dvol_{\og} \]
 for all
 \[ u\in\mC_{f,\phi,\psi} := \left\{ v\in C^\infty(X) \suchthat B_0^5(v)=f, B_1^5(v)=\phi, B_2^5(v)=\psi \right\} , \]
 with equality if and only if $u$ is the solution of~\eqref{eqn:L6_extension}.
\end{thm}

\begin{proof}
 Fix $u_0\in\mC_{f,\phi,\psi}$ and note that, by the linearity of the boundary operators $B_j^5$, it holds that $\mC_{f,\phi,\psi}=u_0+\mC_{0,0,0}$.  Given any $v\in\mC_{0,0,0}$, it follows from the fact that $\mQ_6$ is bilinear and symmetric that
 \begin{align*}
  \mE_6(u_0+v) & = \mE_6(u_0) + 2\mQ_6(v,u_0) + \mE_6(v) \\
  & = \mE_6(u_0) + \mE_6(v) + 2\int_X v\,L_6u_0 \\
  & \geq \mE_6(u_0) + \lambda_{1,D}(L_6)\lV v\rV_{L^2(X)}^2 - 2\lV v\rV_{L^2(X)}\lV L_6u_0\rV_{L^2(X)} \\
  & \geq \mE_6(u_0) - \frac{1}{\lambda_{1,D}(L_6)}\lV L_6u_0\rV_{L^2(X)} .
 \end{align*}
 In particular, we see that $\mE_6$ is bounded below on $\mC_{f,\phi,\psi}$.  Note also that the second equality in the above display implies that
 \begin{equation}
  \label{eqn:dirichlet_critical_equation}
  \left.\frac{d}{dt}\right|_{t=0}\mE_6(u_0+tv) = 2\int_X v\,L_6u_0 .
 \end{equation}
 In particular, critical points of $\mE_6\colon\mC_{f,\phi,\psi}\to\bR$ are solutions of~\eqref{eqn:L6_extension}.
 
 Now, since $(L_6;B_0^5,B_1^5,B_2^5)$ satisfies the Lopatinskii--Shapiro conditions~\cite{AgmonDouglisNirenberg1959,Lopatinskii1953,Shapiro1953}, we may use~\eqref{eqn:dirichlet_critical_equation} to conclude that a minimizing solution converges to a solution of~\eqref{eqn:L6_extension}.  Finally, the assumption $\lambda_{1,D}(L_6)>0$ implies that $\ker_D L_6=\{0\}$, and hence this solution is unique.
\end{proof}

Combining this with our understanding of $\mB_j^5$ on Poincar\'e--Einstein manifolds yields the proof of Theorem~\ref{thm:L2-trace}.

\begin{proof}[Proof of Theorem~\ref{thm:L2-trace}]
 This is an immediate consequence of Theorem~\ref{thm:dirichlet-to-neumann_restatement} and Theorem~\ref{thm:genl-L2-trace}.
\end{proof}

As discussed in the introduction, Theorem~\ref{thm:genl-L2-trace} can be combined with a sharp norm inequality for the embedding $W^{k,2}(M^n)\hookrightarrow L^{\frac{2n}{n-2k}}(M^n)$ to obtained a sharp norm inequality for the embedding~\eqref{eqn:Lp-trace-embedding}.  A conformally invariant sharp norm inequality for this embedding is known on the round sphere~\cite{Beckner1993,Lieb1983}.  Thus this approach bears fruit for any manifold with {\coronal} boundary which is conformal to the round sphere.

\begin{prop}
 \label{prop:Lp-trace-subcritical}
 Let $(X^{n+1},g)$, $n\geq6$, be a compact Riemannian manifold with {\coronal} boundary $(M^n,h)$.  Suppose that $(M^n,h)$ is conformally equivalent to the round $n$-sphere.  Then
 \begin{equation}
  \label{eqn:Lp-trace-subcritical}
  \mE_6(u) \geq \frac{8}{3}C_{n,5/2}\lV f\rV_{\frac{2n}{n-5}}^2 + 8C_{n,3/2}\lV\phi\rV_{\frac{2n}{n-3}}^2 + 3C_{n,1/2}\lV\psi\rV_{\frac{2n}{n-1}}^2
 \end{equation}
 for all $u\in C^\infty(X)$, where $f:=B_0^5(u)$, $\phi:=B_1^5(u)$, $\psi:=B_2^5(u)$, the $L^p$-norms on the right-hand side are taken with respect to the Riemannian volume element $\dvol_h$, and $C_{n,\gamma}$ is given by~\eqref{eqn:sobolev_constants}.  Moreover, equality holds in~\eqref{eqn:Lp-trace-subcritical} if and only if $L_6u=0$ and there are real numbers $a_1,a_2,a_3\in\bR$ and points
 \[ x_1,x_2,x_3 \in B^{n+1} := \left\{ x\in\bR^{n+1} \suchthat \lv x\rv<1 \right\} \]
 such that
 \begin{equation}
  \label{eqn:Lp-trace-subcritical-rigidity} 
  \begin{split}
   f(x) & = a_1\left((1 + x\cdot x_1)e^{\Upsilon(x)}\right)^{-\frac{n-5}{2}} , \\
   \phi(x) & = a_2\left((1 + x\cdot x_2)e^{\Upsilon(x)}\right)^{-\frac{n-3}{2}} , \\
   \psi(x) & = a_3\left((1 + x\cdot x_3)e^{\Upsilon(x)}\right)^{-\frac{n-1}{2}} ,   
  \end{split}
 \end{equation}
 where $\Upsilon\in C^\infty(M)$ is such that $e^{-2\Upsilon}h$ is a round metric on $S^n$ and $S^n$ is regarded as the boundary of $B^{n+1}$.
\end{prop}

\begin{proof}
 The sharp Sobolev inequality~\cite{Beckner1993} states that for every $\gamma\in(0,n/2)$ and every $w\in C^\infty(S^n)$, it holds that
 \begin{equation}
  \label{eqn:sobolev-sphere}
  \int_{S^n} w\,P_{2\gamma}w \geq C_{n,\gamma}\left(\int_{S^n} \lv w\rv^{\frac{2n}{n-2\gamma}}\right)^{\frac{n-2\gamma}{n}}
 \end{equation}
 where $C_{n,\gamma}$ is given by~\eqref{eqn:sobolev_constants}.  Moreover, equality holds if and only if $w(x)=a(1+x\cdot x_0)^{-\frac{n-2\gamma}{2}}$ for some $a\in\bR$ and $x_0\in B^{n+1}$.
 
 Since $(M^n,h)$ is conformal to the round $n$-sphere $(S^n,d\theta^2)$, there is an $\Upsilon\in C^\infty(M)$ such that $h=e^{2\Upsilon}d\theta^2$.  By the conformal covariance of the fractional GJMS operators, $P_{2\gamma}^hw = e^{-\frac{n+2\gamma}{2}\Upsilon}P_{2\gamma}^{d\theta^2}\left(e^{\frac{n-2\gamma}{2}\Upsilon}w\right)$ for all $\gamma\in(0,n/2)$.  Therefore~\eqref{eqn:sobolev-sphere} implies that for all $w\in C^\infty(M)$,
 \begin{equation}
  \label{eqn:sobolev-conformal-sphere}
  \oint_M w\,P_{2\gamma}w \geq C_{n,\gamma}\left(\oint_{M} \lv w\rv^{\frac{2n}{n-2\gamma}}\right)^{\frac{n-2\gamma}{n}}
 \end{equation}
 with equality if and only if $e^{\frac{n-2\gamma}{2}\Upsilon}w(x)=a(1+x\cdot x_0)^{-\frac{n-2\gamma}{2}}$ for some $a\in\bR$ and $x_0\in B^{n+1}$.  Combining~\eqref{eqn:sobolev-conformal-sphere} with Theorem~\ref{thm:genl-L2-trace} yields the desired conclusion.
\end{proof}

In the critical case $\dim \partial X=5$, the sharp Onofri inequality~\cite{Beckner1993}, which provides a sharp norm inequality for the embedding $W^{5/2,2}(S^5)\hookrightarrow e^L(S^5)$, enables us to extend Proposition~\ref{prop:Lp-trace-subcritical} to the critical dimension.  For simplicity, we assume in this case that the boundary is isometric to the round five-sphere; one can also apply conformal covariance to obtain a similar result when the boundary is only conformal to the round five-sphere (cf.\ Corollary~\ref{cor:Lp-trace-critical-upper}).

\begin{prop}
 \label{prop:Lp-trace-critical}
 Let $(X^6,g)$ be a compact Riemannian manifold with {\coronal} boundary $(M^5,h)$.  Suppose that $(M^5,h)$ is isometric to the round five-sphere.  Then
 \begin{equation}
  \label{eqn:Lp-trace-critical}
  \mE_6(u) \geq 3C_{5,1/2}\lV\psi\rV_{5/2}^2 + 8C_{5,3/2}\lV\phi\rV_{5}^2 + \frac{128}{5}\Vol(S^5)\ln\oint_M e^{5(f-\bar f)}\,d\mu
 \end{equation}
 for all $u\in C^\infty(X)$, where $f:=B_0^5(u)$, $\phi:=B_1^5(u)$, $\psi:=B_2^5(u)$, the last integral is taken with respect to $d\mu:=\frac{1}{\Vol(M)}\dvol_h$, the remaining integrals on the right-hand side are taken with respect to the Riemannian volume element $\dvol_h$ and $\bar f$ is the average of $f$ with respect to $\dvol_h$, and $C_{n,\gamma}$ is given by~\eqref{eqn:sobolev_constants}.  Moreover, equality holds in~\eqref{eqn:Lp-trace-critical} if and only if $L_6u=0$ and there are real numbers $a_1,a_2,a_3\in\bR$ and points $x_1,x_2,x_3\in B^6$ such that
 \begin{equation}
  \label{eqn:Lp-trace-critical-rigidity} 
  \begin{split}
   f(x) & = a_1 - \ln (1+x\cdot x_1) , \\
   \phi(x) & = a_2\left(1 + x\cdot x_2\right)^{-1} , \\
   \psi(x) & = a_3\left(1 + x\cdot x_3\right)^{-2} ,   
  \end{split}
 \end{equation}
 where $S^5$ is regarded as the boundary of $B^6$.
\end{prop}

\begin{proof}
 The sharp Onofri inequality~\cite{Beckner1993} states that on the round $n$-sphere, it holds that
 \begin{equation}
  \label{eqn:onofri-sphere}
  \int_{S^n} w\,P_nw \geq \frac{2(n-1)!}{n}\Vol(S^n) \ln \int_{S^n} e^{n(w-\bar w)}\,d\mu
 \end{equation}
 for all $w\in C^\infty(S^n)$, where $\bar w = (\int w)/\int 1)$ is the average of $w$, the integral on the left-hand side and the volume are taken with respect to the Riemannian volume element $\dvol$ on $S^n$, and $d\mu:=\frac{1}{\Vol(S^n)}\dvol$.  Moreover, equality holds in~\eqref{eqn:onofri-sphere} if and only if $w(x)=a-\ln(1+x\cdot x_0)$ for some $a\in\bR$ and $x_0\in B^6$.  Combining~\eqref{eqn:sobolev-sphere} and~\eqref{eqn:onofri-sphere} with Theorem~\ref{thm:genl-L2-trace} yields the desired conclusion.
\end{proof}

We now apply Proposition~\ref{prop:Lp-trace-subcritical} and Proposition~\ref{prop:Lp-trace-critical} to Euclidean upper half space, the closed Euclidean ball, and the round hemisphere to obtain the explicit sharp Sobolev inequalities from the Introduction.  We begin with the noncritical cases.

\begin{proof}[Proof of Corollary~\ref{cor:Lp-trace-upper}]
 It is well-known that $(\bR_+^{n+1},dx^2+dy^2)$ is a compactification of the upper half-space model $\left(\bR_+^{n+1},y^{-2}(dx^2+dy^2)\right)$ of hyperbolic space, and as such is conformally equivalent to the closed Euclidean ball $B^{n+1}$.  Moreover, the metric $dx^2$ on $\bR^n := \partial\bR_+^{n+1}$ is conformally equivalent to the round metric on $S^n$ via $d\theta^2=(1+\lv x\rv^2)^{-2}dx^2$.  In particular, we may apply Proposition~\ref{prop:Lp-trace-subcritical} to deduce that~\eqref{eqn:Lp-trace-subcritical} holds.
 
 It remains to express $\mE_6$ in the form of the conclusion of Corollary~\ref{cor:Lp-trace-upper}.  To that end, note that since $(\bR_+^{n+1},dx^2+dy^2)$ is flat, it holds that $L_6=(-\Delta)^3$.  Moreover, $(\bR^n,dx^2)$ is totally geodesic, and hence {\coronal} and flat.  We conclude that
 \begin{align*}
  B_0^5(u) & = u, \\
  B_1^5(u) & = \eta u, \\
  B_2^5(u) & = \Delta u - \frac{4}{3}\oDelta u, \\
  B_3^5(u) & = \eta\Delta u - 4\oDelta\eta u , \\
  B_4^5(u) & = -\Delta^2u - 4\oDelta\Delta u + 8\oDelta^2u , \\
  B_5^5(u) & = \eta\Delta^2u + \frac{4}{3}\oDelta\eta\Delta u + \frac{8}{3}\oDelta^2\eta u .
 \end{align*}
 On the one hand, the fact that $\eta=-\partial_y$ and $\nabla_\eta\eta=0$ implies that $B_1^5=-\partial_y$ and $B_2^5=\partial_y^2-\frac{1}{3}\oDelta$.  On the other hand, integration by parts yields
 \[ \mE_6(u) = \int_{\bR_+^{n+1}} \lv\nabla\Delta u\rv^2 + \oint_M \left\{ 16u\oDelta^2\eta u - 8(\eta u)\oDelta\Delta u \right\} . \]
 Combining this display with~\eqref{eqn:Lp-trace-subcritical} and the formulas for $f$, $\phi$, and $\psi$ yields the final conclusion.
\end{proof}

\begin{proof}[Proof of Corollary~\ref{cor:Lp-trace-disk}]
 It is clear that we may apply Proposition~\ref{prop:Lp-trace-subcritical} to $(B^{n+1},dx^2)$, leaving us to only compute $\mE_6$ in the form of the conclusion of Corollary~\ref{cor:Lp-trace-disk}.  To that end, note that since $(B^{n+1},dx^2)$ is flat, it holds that $L_6=(-\Delta)^3$.  Moreover, $\partial B^{n+1}$ is {\coronal} with constant mean curvature $H=n$.  It follows from Lemma~\ref{lem:gauss-codazzi} that $\oP=\frac{1}{2}\og$.  Therefore
 \begin{align*}
  B_0^5(u) & = u, \\
  B_1^5(u) & = \eta u + \frac{n-5}{2}u , \\
  B_2^5(u) & = \Delta u - \frac{4}{3}\oDelta u - 4\eta u + \frac{(n-3)(n-5)}{3}u , \\
  B_3^5(u) & = \eta\Delta u - 4\oDelta\eta u + \frac{n-9}{2}\Delta u - 2(n-7)\oDelta u + (n^2-2n+9)\eta u + 4\frac{\Gamma\bigl(\frac{n+1}{2}\bigr)}{\Gamma\bigl(\frac{n-5}{2}\bigr)}u , \\
  B_4^5(u) & = -\Delta^2u - 4\oDelta\Delta u + 8\oDelta^2u + 4\eta\Delta u + 16\oDelta\eta u + (n-3)(n+3)\Delta u \\
   & \quad - 4(n^2-4n+1)\oDelta u - 4(n-3)(n+1)\eta u + 8\frac{\Gamma\bigl(\frac{n+3}{2}\bigr)}{\Gamma\bigl(\frac{n-5}{2}\bigr)}u , \\
  B_5^5(u) & = \eta\Delta^2u + \frac{4}{3}\oDelta\eta\Delta u + \frac{8}{3}\oDelta^2\eta u + \frac{n-5}{2}\Delta^2u + \frac{2(n-3)}{3}\oDelta\Delta u + \frac{4(n-1)}{3}\oDelta^2u \\
   & \quad - \frac{(n-5)(n+3)}{3}\eta\Delta u - \frac{4(n^2-2n-9)}{3}\oDelta\eta u - \frac{(n-5)(n-3)(n+3)}{6}\Delta u \\
   & \quad - \frac{2(n-1)(n^2-2n-9)}{3}\oDelta u + \frac{(n-5)(n-3)(n+1)(n+3)}{6}\eta u + \frac{8}{3}\frac{\Gamma\bigl(\frac{n+5}{2}\bigr)}{\Gamma\bigl(\frac{n-5}{2}\bigr)}u .
 \end{align*}
 Using these expressions and integrating by parts yields
 \[ \mE_6(u) = \int_X \lv\nabla\Delta u\rv^2 + \oint_M \bigl\{ (\Delta u)\,A(u) + (\eta u)\,B(u) + u\,C(u) \bigr\} , \]
 where
 \begin{align*}
  A(u) & = \frac{n-9}{2}\Delta u - 8\oDelta\eta u - 4(n-7)\oDelta u + 2(n^2-2n+9)\eta u + 8\frac{\Gamma\bigl(\frac{n+1}{2}\bigr)}{\Gamma\bigl(\frac{n-5}{2}\bigr)}u, \\
  B(u) & = 32\oDelta\eta u - 8(n^2-2n+3)\eta u + 16\oDelta^2u \\
   & \qquad - 8(n^2-6n+15)\oDelta u + 8(n-3)\frac{\Gamma\bigl(\frac{n+1}{2}\bigr)}{\Gamma\bigl(\frac{n-5}{2}\bigr)}u , \\
  C(u) & = 8(n-5)\oDelta^2u - 4(n-3)(n^2-6n+7)\oDelta u + 4(n-1)(n-3)\frac{\Gamma\bigl(\frac{n+1}{2}\bigr)}{\Gamma\bigl(\frac{n-5}{2}\bigr)}u .
 \end{align*}
 Rewriting this in terms of $B_0^5(u)$, $B_1^5(u)$, and $B_2^5(u)$ and combining it with~\eqref{eqn:Lp-trace-subcritical} yields the final conclusion.
\end{proof}

\begin{proof}[Proof of Corollary~\ref{cor:Lp-trace-hemisphere}]
 It is clear that we may apply Proposition~\ref{prop:Lp-trace-subcritical} to $(S_+^{n+1},d\theta^2)$, leaving us only to compute $\mE_6$ in the form of the conclusion of Corollary~\ref{cor:Lp-trace-hemisphere}.  To that end, note that since $d\theta^2$ is the round metric with $\Ric=n\,d\theta^2$, it holds that
 \[ L_6 = \left(-\Delta + \frac{(n+1)(n-1)}{4}\right)\left(-\Delta + \frac{(n+3)(n-3)}{4}\right)\left(-\Delta + \frac{(n+5)(n-5)}{4}\right) . \]
 Moreover, $\partial S_+^{n+1}$ is {\coronal} with totally geodesic boundary.  It follows from Lemma~\ref{lem:gauss-codazzi} that $\oP=\frac{1}{2}\og$.  Therefore
 \begin{align*}
  B_0^5(u) & = u, \\
  B_1^5(u) & = \eta u, \\
  B_2^5(u) & = \Delta u - \frac{4}{3}\oDelta u + \frac{(n-3)(n-5)}{12}u, \\
  B_3^5(u) & := \eta\Delta u - 4\oDelta\eta u + \frac{3n^2-8n+13}{4}\eta u, \\
  B_4^5(u) & := -\Delta^2u - 4\oDelta\Delta u + 8\oDelta^2u + \frac{3n^2-4n-11}{2}\Delta u - (3n+1)(n-3)\oDelta u \\
   & \quad + \frac{3(n+1)(n-1)(n-3)(n-5)}{16}u, \\
  B_5^5(u) &= \eta\Delta^2u + \frac{4}{3}\oDelta\eta\Delta u + \frac{8}{3}\oDelta^2\eta u - \frac{5n^2-4n-45}{6}\eta\Delta u - \frac{5n^2-8n-37}{3}\oDelta\eta u \\
    & \quad + \frac{(n+3)(n+1)(15n^2-100n+149)}{48}\eta u.
 \end{align*}
 Using these expressions and integrating by parts yields
 \begin{align*}
  \mE_6(u) & = \int_X \biggl\{ \lv\nabla\Delta u\rv^2 + \frac{3n^2-35}{4}(\Delta u)^2 + \frac{3n^4-70n^2+259}{16}\lv\nabla u\rv^2 + \frac{\Gamma\bigl(\frac{n+7}{2}\bigr)}{\Gamma\bigl(\frac{n-5}{2}\bigr)}u^2 \biggr\} \\
   & \quad + \oint_M \biggl\{ \left(-8\oDelta\eta u + \frac{3n^2-8n+13}{2}\eta u\right)\Delta u \\
    & \qquad + \left(16\oDelta^2u - 2(3n^2-8n-3)\oDelta u + 6\frac{\Gamma\bigl(\frac{n+3}{2}\bigr)}{\Gamma\bigl(\frac{n-5}{2}\bigr)}u\right)\eta u \biggr\} . 
 \end{align*}
 Rewriting this in terms of $B_0^5(u)$, $B_1^5(u)$, and $B_2^5(u)$ and combining it with~\eqref{eqn:Lp-trace-subcritical} yields the final conclusion.
\end{proof}

We conclude with the critical cases.

\begin{proof}[Proof of Corollary~\ref{cor:Lp-trace-critical-upper}]
 From the proof of Corollary~\ref{cor:Lp-trace-upper}, we see that the six-dimensional upper half space is such that
 \[ \mE_6(u) = \int_{\bR_+^6} \lv\nabla\Delta u\rv^2 + \oint_{\partial\bR_+^6} \left\{ 8\lp\onabla\psi,\onabla\phi\rp + \frac{16}{3}(\oDelta\phi)(\oDelta f) \right\} \]
 for all $u\in C^\infty(\bR_+^6)\cap W^{3,2}(\bR_+^6)$, where $f:=B_0^5(u)$, $\phi:=B_1^5(u)$, and $\psi:=B_2^5(u)$.  Using the identity $d\theta^2=\bigl(\frac{1+\lv x\rv^2}{2}\bigr)^{-2}dx^2$, we see that in Euclidean space, the Onofri inequality~\eqref{eqn:onofri-sphere} becomes
 \[ \int_{\bR^n} w\,P_nw \geq \frac{2(n-1)!}{n}\Vol(S^n)\ln\int_{\bR^n} e^{n(w-\bar w)}\,d\mu , \]
 where $\ow$ is the average of $w$ with respect to $d\mu:=\frac{1}{\Vol(S^n)}\bigl(\frac{1+\lv x\rv^2}{2}\bigr)^{-n}\dvol_{dx^2}$.  Moreover, equality holds if and only if
 \[ w(x) = a - \ln\left(\varepsilon + \lv x-x_0\rv^2\right) + \ln\left(1 + \lv x\rv^2\right) \]
 for some $a,\varepsilon\in\bR$ and some $x_0\in\bR^n$.  The conclusion now follows from Proposition~\ref{prop:Lp-trace-critical}.
\end{proof}

\begin{proof}[Proof of Corollary~\ref{cor:Lp-trace-critical-disk}]
 From the proof of Corollary~\ref{cor:Lp-trace-disk}, we see that the six-dimensional closed Euclidean ball is such that
 \begin{multline*}
  \mE_6(u) = \int_{B^6} \lv\nabla\Delta u\rv^2 + \oint_{\partial B^6} \biggl\{ -2\psi^2 + 8\lp\onabla\psi,\onabla\phi\rp + 32\psi\phi - \frac{8}{3}\lp\onabla\psi,\onabla f\rp  \\ + 16\phi^2 + \frac{16}{3}(\oDelta\phi)(\oDelta f) + \frac{16}{3}\lp\onabla\phi,\onabla f\rp + \frac{64}{9}(\oDelta f)^2 + 16\lv\onabla f\rv^2\biggr\}
 \end{multline*}
 for all $u\in C^\infty(B^6)$, where $f:=B_0^5(u)$, $\phi:=B_1^5(u)$, and $\psi:=B_2^5(u)$.  The conclusion now follows from Proposition~\ref{prop:Lp-trace-critical}.
\end{proof}

\begin{proof}[Proof of Corollary~\ref{cor:Lp-trace-critical-hemisphere}]
 From the proof of Corollary~\ref{cor:Lp-trace-hemisphere}, we see that the six-dimensional upper hemisphere is such that
 \begin{align*}
  \mE_6(u) & = \int_{S_+^6} \left\{ \lv\nabla\Delta u\rv^2 + 10(\Delta u)^2 + 24\lv\nabla u\rv^2\right\} \\
   & \quad + \oint_{\partial S_+^6} \left\{ 8\lp\onabla\psi,\onabla\phi\rp + 24\psi\phi + \frac{16}{3}(\oDelta\phi)(\oDelta f) + 32\lp\onabla f,\onabla\phi\rp \right\}
 \end{align*}
 for all $u\in C^\infty(S_+^6)$, where $f:=B_0^5(u)$, $\phi:=B_1^5(u)$, and $\psi:=B_2^5(u)$.  The conclusion now follows from Proposition~\ref{prop:Lp-trace-critical}.
\end{proof}

\bibliographystyle{abbrv}
\bibliography{bib}
\end{document}